\documentclass[11pt]{amsart}
\usepackage{amssymb, amsmath, amsthm}
\usepackage[colorlinks=true,linkcolor=blue,citecolor=red]{hyperref}
\usepackage{color}
\usepackage{epsfig}
\usepackage{amsmath}
\usepackage{amssymb}
\usepackage{amscd}
\usepackage{graphicx}
\usepackage{mathrsfs}
\usepackage{enumerate}
\usepackage{tikz}
\usepackage{epstopdf}
\usepackage[T1]{fontenc}
\numberwithin{equation}{section}

   \newtheorem{thm}{Theorem}[section]
   
   \newtheorem{prop}{Proposition}[section]

\usepackage{epsfig}

\newtheorem{rem}[thm]{Remark}
\newtheorem{lem}[thm]{Lemma}
\newtheorem{assum}[thm]{Assumption}
\newtheorem{RHP}[thm]{Riemann-Hilbert Problem}

\numberwithin{equation}{section}
\numberwithin{prop}{section}
\numberwithin{lemma}{section}
\numberwithin{re}{section}
\numberwithin{coro}{section}



\thanks{ Email: sftian@cumt.edu.cn, shoufu2006@126.com (S. F. Tian). }

\begin{document}

\title[Long time asymptotic for the WKI equation with finite density initial data]{Long time asymptotic for the Wadati-Konno-Ichikawa equation with  finite density initial data}


\author[Li]{Zhi-Qiang Li}

\author[Tian]{Shou-Fu Tian$^{*}$}
\address{Zhi-Qiang Li, Shou-Fu Tian (Corresponding author) and Jin-Jie Yang \newline
School of Mathematics, China University of Mining and Technology, Xuzhou 221116, People's Republic of China}
\thanks{$^{*}$Corresponding author(sftian@cumt.edu.cn, shoufu2006@126.com). 
}
\email{sftian@cumt.edu.cn, shoufu2006@126.com (S.F. Tian)}

\author[Yang]{Jin-Jie Yang}

\begin{abstract}
{In this work, we investigate
the Cauchy problem of the Wadati-Konno-Ichikawa (WKI) equation with finite density initial data.
Employing the $\bar{\partial}$-generalization of Deift-Zhou nonlinear steepest descent method, we derive  the long time asymptotic behavior of the solution $q(x,t)$ in space-time soliton region.
Based on the resulting asymptotic behavior, the asymptotic approximation of the WKI equation is characterized with the soliton term confirmed by $N(I)$-soliton on discrete spectrum and the $t^{-\frac{1}{2}}$ leading order term on continuous spectrum with residual error up to $O(t^{-\frac{3}{4}})$. Our results also confirm the soliton resolution  conjecture for the WKI equation.
}
\end{abstract}

\maketitle

\tableofcontents

\section{Introduction}

It is well-known that the nonlinear Schr\"{o}dinger (NLS) equation,
\begin{align}\label{NLS}
       iu_{t}\pm u_{xx}+2|u|^{2}u=0,
\end{align}
can be adapted to describe the pulse propagation in optical fibers\cite{I-6}. As a slowly
varying amplitude approximation of Maxwell's equations, the NLS equation can be used to describe the propagation of these very narrow pulseless. Moreover,  the NLS equation has widespread application in deep water waves\cite{I-1}, plasma physics\cite{I-2,I-3}, magneto-static spin waves\cite{I-8} and so on. Therefore, because of the  significance of the NLS equation,  more and more scholars make their contributions to to the research on the NLS equation and its extensions  \cite{Miller-1,Tian-PAMS,Tian-PA,Tian-JDE,Wangds-2019-JDE}.

However, the NLS equation is not always applicable in actual physical scenarios. For example, it is not appropriate to use the NLS equation to describe the propagation of soliton in materials with saturation effects. At higher field strength, the optically induced refractive-index change becomes saturated.  The saturation effects will result in a physical limit of the shortest soliton pulse duration or of the pulse compression by high-order soliton generation.  In order to investigate the the propagation of soliton in materials with saturation effects, Herrmann, Porsezian, etc.\cite{A-equation-1,A-equation-2} proposed and studied the equation
\begin{align}\label{A-WKI}
iA_t+A_{xx}+\frac{|A|^{2}}{1+\gamma|A|^{2}}A=0,
\end{align}
where $A$ is the slowly varying amplitude of the field strength and $\gamma$ is the Kerr parameter.
However, because the model \eqref{A-WKI} is not integrable, it is difficult for scholars to derive the analytical solutions of the model. The earliest research on integrable model possessing saturation effects can be traced back to the work of Wadati, Konno and Ichikawa \cite{WKI-1}. They proposed the following  integrable model,
\begin{align}\label{WKI-equation}
iq_{t}+\left(\frac{q}{\sqrt{1+|q|^{2}}}\right)_{xx}=0,
\end{align}
which is later called the Wadati-Konno-Ichikawa (WKI) equation.
Then, Wadati, Konno and Ichikawa  did a series of work on WKI equation, including the infinite number of conservation laws and an application that it can be used to describe nonlinear transverse oscillations of elastic beams under tension \cite{WKI-2,WKI-3,WKI-conservation-law}. Then, in 2005, from the motions of curves in Euclidean geometry $E^{3}$, Qu and Zhang \cite{WKI-4} derived the  WKI equation \eqref{WKI-equation}.

Notably, in nonlinear science, the WKI equation \eqref{WKI-equation} is a very special integrable model. In fact, we know that the focusing nonlinear integrable systems have bright soliton from vacuum, while the de-focusing nonlinear integrable systems possess dark soliton from nonzero background \cite{Ablowitz-f-d,Zakharov-f-d}. However, for both of the focusing case ($r=-\bar{q}$,in \eqref{WKI-equation} ) and  the defocusing case ($r=-\bar{q}$,in \eqref{WKI-equation}), the WKI equation \eqref{WKI-equation} possesses  bright soliton from vacuum \cite{Zhang-f-d}. Moreover, the modulation instability has been universally used to explain the occurrence of rogue waves appearing from nowhere in the ocean. It is interesting that the WKI equation \eqref{WKI-equation} possesses rogue wave solution but possesses no modulation instability \cite{Zhang-f-d}. Additionally, the other aspects of the WKI equation \eqref{WKI-equation} including  the orbital stability \cite{WKI-5}, algebra-geometric constructions \cite{WKI-6}, the existence of global solution \cite{WKI-7} and Riemann-Hilbert (RH) method also have been considered \cite{WKI-7+1,WKI-7+2}. What's more, by introducing a hodograph transformation, the WKI equation \eqref{WKI-equation} can be transformed into an equivalent form(called modified WKI equation) which has been studied with zero and nonzero boundary conditions by applying Darboux transformation \cite{Zhang-f-d}.

In this work, we 
investigate the long time asymptotic behavior for the Cauchy problem  of the WKI equation \eqref{WKI-equation} with finite density  initial data
\begin{align}\label{boundary}
q(x,0)=q_{0}(x),~~\lim_{x\rightarrow\pm\infty}q_{0}(x)=q_{\pm},~~|q_{\pm}|=q_{0}.
\end{align}
where $q_{0}$ is a nonzero real constant.

The research on the long time asymptotic behavior of nonlinear wave equations which is solvable by the inverse scattering method, can be traced to the earlier work of Manakov \cite{Manakov-1974}. In 1976, Zakharov and Manakov \cite{Zakharov-1976} derived the first result for the long time asymptotic solutions of NLS equation with decaying initial value.  In 1993, a nonlinear steepest descent method was  developed by Defit and Zhou \cite{Deift-1993} to rigorously derive the
long time asymptotic behavior of the solution for the modified Korteweg-de Vries (mKdV) equation equation via deforming contours to transform the original RH problem into a model problem whose solution is derived in terms of parabolic cylinder model. Since then  this method was widely adapted to other integrable systems. After the unremitting efforts of scholars, the Defit-Zhou steepest descent method has been improved and widely applied \cite{Biondini-fNLS-CMP,Boutet-fNLS-CMP,Yan-fHirota-MPAG,Liu-SS-JMP,Xu-SP-JDE}. The work \cite{Deift-1994-1, Deift-1994-2} showed that if the initial value is smooth and decays fast enough, the error term is $O(\frac{\log t}{t})$. While, in \cite{Deift-2003}, the error term could reach $O(t^{-(\frac{1}{2}+\iota)})$ for any $0<\iota<\frac{1}{4}$, under the condition that the initial value belongs to the weighted Sobolev space \eqref{W-H-Sobolev-spaces}.

In recent years, McLaughlin and Miller \cite{McLaughlin-1, McLaughlin-2} presented a $\bar{\partial}$-steepest descent method by combining steepest descent with $\bar{\partial}$-problem to  study the asymptotic of orthogonal polynomials with non-analytical weights.  Then, scholars developed this method to investigate defocusing NLS equation with finite mass initial data \cite{Dieng-2008} and with finite density initial data \cite{Cuccagna-2016}.
As the widespread application of $\bar{\partial}$-steepest descent method, some advantages have displayed. For example, the delicate estimates involving $L^{p}$ estimates of Cauchy projection operators can be avoided during the analysis. In addition, the work \cite{Dieng-2008} showed an improvement that the error term reach $O(t^{-\frac{3}{4}})$ when the initial value belongs to the weighted Sobolev space. This $\bar{\partial}$-steepest descent method was also successfully applied to prove asymptotic stability of $N$-soliton solutions to focusing NLS equation \cite{AIHP}. Therefore, a series of great work has been done by using $\bar{\partial}$-steepest descent method \cite{Faneg-2,Miller-2,Jenkins,Jenkins2,Li-CSP-Dbar,Li-cgNLS,Fan-mKdV-Dbar-1, Fan-SP-Dbar,Fan-mCH-Dbar,Yang-Hirota,Fan-mKdV-Dbar-2}.

In \cite{Geng-pWKI-JMAA}, Cheng, Geng and  Wang have analysed the long time asymptotic solutions of potential WKI equation by using  Defit-Zhou steepest descent method. Here, 
we extend above results to study the long time asymptotic behavior of the solution $q(x,t)$ of the WKI equation \eqref{WKI-equation} by using $\bar{\partial}$-steepest descent method  with finite density initial data. Moreover, our work shows that the accuracy of our result can reach $O(t^{-\frac{3}{4}})$, this is barely possible for literature \cite{Geng-pWKI-JMAA}.
\\

\noindent \textbf{Organization of the rest of the work}

In section 2, based on the Lax pair of the WKI equation, we introduce two kinds of eigenfunctions to deal with the spectral singularity. Moreover, the analyticity, symmetries and asymptotic properties are analyzed.

In section 3, by using similar ideas to \cite{Xu-SP-JDE}, the RHP for $M(z)$ is constructed for the WKI equation with finite density initial data.

In section 4,  we introduce a set of conjugations
and interpolations transformation, such that $M^{(1)}(z)$ becomes a standard RH problem.

In section 5, according to the factorization of the jump matrix of the RH problem on the
real line, we make the continuous extension of the jump matrix off the real axis by introducing a matrix function $R^{(2)}(z)$ and get a mixed $\bar{\partial}$-RH problem.

In section 6, we decompose the mixed $\bar{\partial}$-RH problem into two parts that are a model RH problem for $M^{(2)}(z)$ with $\bar{\partial}R^{(2)}=0$ and a pure $\bar{\partial}$-RH problem for $M^{(3)}(z)$ with $\bar{\partial}R^{(2)}\neq0$. Furthermore,
we solve the model RH problem $M^{(2)}_{R}$  via an  modified reflectionless RH model $M^{sol}(z)$ for the soliton part in section 7. and an local solvable model near the phase point $z_{0}$ which can be solved by matching  parabolic cylinder model problem in section 8.

In section 9, the error function $E(z)$ with a small-norm RH problem is achieved. 

In section 10, the pure $\bar{\partial}$-RH problem for $M^{(3)}$ is investigated.

Finally, in  section 11,  we obtain the soliton resolution and long time asymptotic behavior of the WKI equation with finite density initial data.

\section{The spectral analysis of WKI equation}

The  WKI equation admits the Lax pair
\begin{align}\label{O-Lax}
\phi_{x}=X\phi,~~\phi_{t}=T\phi,
\end{align}
where $X=-ik\sigma_{3}+kQ$,
\begin{align}\label{Lax-1}
T=\left(
    \begin{array}{cc}
      -\frac{2ik^{2}}{\Phi} & \frac{2qk^{2}}{\Phi}+ik\left(\frac{q}{\Phi}\right)_{x} \\
      -\frac{2\bar{q}k^{2}}{\Phi}+ik\left(\frac{\bar{q}}{\Phi}\right)_{x} & \frac{2ik^{2}}{\Phi} \\
    \end{array}
  \right), ~~Q=\left(
                 \begin{array}{cc}
                   0 & q \\
                   -\bar{q} & 0 \\
                 \end{array}
               \right),~~\Phi=\sqrt{1+|q|^{2}},
\end{align}
and $\sigma_{3}$ is the third Pauli matrix:
\begin{align*}
\sigma_{1}=\left(
  \begin{array}{cc}
    0 & 1 \\
    1 & 0 \\
  \end{array}
\right),~~\sigma_{2}=\left(
  \begin{array}{cc}
    0 & -i \\
    i & 0 \\
  \end{array}
\right), ~~\sigma_{3}=\left(
  \begin{array}{cc}
    1 & 0 \\
    0 & -1 \\
  \end{array}
\right).
\end{align*}
The $\bar{q}$ implies the complex conjugate of $q$.

Generally, in order to study the initial value problem (IVP) for the WKI equation, we employ the $x$-part of the Lax pair to study the long time asymptotic behaviors. The $t$-part of Lax pair is used to control the time evolution of the scattering data based on the inverse scattering transform method. Similar to the modified Camassa-Holm (mCH) equation  \cite{Fan-mCH-Dbar} and the short pulse equation \cite{Fan-SP-Dbar}, we need to  consider the spectral
problem as $k\rightarrow\infty$ and $k\rightarrow0$  to get the appropriate Lax pair
and the RH problem.

\subsection{The singularity at $k=0$}

Firstly, according to the asymptotic boundary \eqref{boundary}, we consider the asymptotic spectral problem of \eqref{O-Lax}.   Letting $x\rightarrow\pm\infty$, the limit spectral problem is derived  as
\begin{gather}
\phi_{x}=X_{\pm}\phi,~~~ \phi_{t}=T_{\pm}\phi, \label{asy-Lax}
\end{gather}
where
\begin{gather}
X_{\pm}=-ik\sigma_{3}+kQ_{\pm}, ~~
T_{\pm}=\frac{2k}{\Phi}X_{\pm},\notag\\
 Q_{\pm}=\left(                 
  \begin{array}{cc}   
    0 & q_{\pm}  \\  
    -\bar{q}_{\pm} & 0 \\  
  \end{array}
\right).\notag
\end{gather}
A direct calculation shows that the eigenvalues of $X_{\pm}$ are $\pm i\lambda=ik\sqrt{q_{0}^{2}+1}$. As a result, the eigenvalues of $T_{\pm}$ are $\pm i\frac{2k}{\Phi_{0}}\lambda$ where
\begin{align}\label{Phi-0-definition}
\Phi_{0}=\sqrt{1+q_{0}^{2}}
\end{align}
Then, by using the above results, the matrices  $X_{\pm}$ and  $T_{\pm}$ can be transformed to diagonal matrices with the same characteristic matrix, i.e.,
\begin{align}\label{diagonal-1}
\begin{split}
  X_{\pm}(x,t;k)&=Y_{\pm}(-i\lambda \sigma_{3})Y_{\pm}^{-1},\\
  T_{\pm}(x,t;k)&=Y_{\pm}(-i\frac{2k}{\Phi_{0}}\lambda)Y_{\pm}^{-1},
  \end{split}
\end{align}
where
\begin{align}\label{Y}
Y_{\pm}=\left(
  \begin{array}{cc}
     1 & i\frac{q_{\pm}}{\Phi_{0}+1} \\
     i\frac{\bar{q}_{\pm}}{\Phi_{0}+1} & 1 \\
  \end{array}
\right)=\mathbb{I}+ i\frac{1}{\Phi_{0}+1}\sigma_{3}Q_{\pm}.
\end{align}
Substituting  \eqref{diagonal-1} into  \eqref{asy-Lax}, we can derive that
\begin{align}
\phi_{\pm}(x,t;k)=Y_{\pm}e^{-i\theta(k)\sigma_{3}}, ~~x\rightarrow\infty,
\end{align}
where $\theta(k)=\lambda(x+\frac{2k}{\Phi_{0}}t)$. Furthermore, we can define
\begin{align}\label{mu0}
\mu^{0}_{\pm}(x,t;k)=\phi_{\pm}(x,t;k)e^{i\theta(x,t;k)\sigma_{3}}  \sim  Y_{\pm}, \quad x\rightarrow\pm\infty,
\end{align}
which together with \eqref{O-Lax} give
\begin{equation} \label{mu0-Lax}
\begin{split}
&(Y_{\pm}^{-1}\mu^{0}_{\pm}(z))_{x}+i\lambda [\sigma_{3}, Y_{\pm}^{-1}\mu^{0}_{\pm}(z)]=Y_{\pm}^{-1}k\Delta Q_{\pm}\mu^{0}_{\pm}(z), \\
&(Y_{\pm}^{-1}\mu^{0}_{\pm}(z))_{t}+i\lambda\frac{2k}{\Phi_{0}}[\sigma_{3}, Y_{\pm}^{-1}\mu^{0}_{\pm}(z)]=Y_{\pm}^{-1}\Delta T_{\pm}\mu^{0}_{\pm}(z),
 \end{split}
 \end{equation}
where $\Delta Q_{\pm}=Q-Q_{\pm}$ and $\Delta T_{\pm}=T-T_{\pm}$.  Then, the equation \eqref{mu0-Lax} can be written in full derivative form, i.e.,
\begin{align}\label{full-derivative-1}
d(e^{i\theta(k)\hat{\sigma}_{3}}Y_{\pm}^{-1}\mu_{\pm}^{0})=
e^{i\theta(k)\hat{\sigma}_{3}}(Y_{\pm}^{-1}k\Delta Q_{\pm}dx+\Delta T_{\pm}dt)\mu_{\pm}^{0},
\end{align}
where $e^{\hat{\sigma}_{3}}A=e^{\sigma_{3}}Ae^{-\sigma_{3}}$. Moreover, the solutions of \eqref{full-derivative-1} can be derived as Volterra integrals
\begin{align}
\begin{matrix}
\mu^{0}_{\pm}(x,t;k)=Y_{\pm}+
\int_{\pm\infty}^{x}Y_{\pm}e^{-i\lambda(x-y)\hat{\sigma}_{3}}Y_{\pm}^{-1}k\Delta Q_{\pm}\mu^{0}_{\pm}(y,t;k)dy.
\end{matrix}
\end{align}
Based on the above Volterra integrals,  analytical properties of $\mu^{0}_{\pm}$ are derived.

\begin{prop}
It is assumed that $q(x)-q_{+}\in L^{1}(a,+\infty)$ and $q(x)-q_{-}\in L^{1}(-\infty,a)$ hold for any constant $a\in \mathbb{R}$. Then, $\mu^{0}_{-,1}, \mu^{0}_{+,2}$ are analytic in $\mathbb{C}^{+}$ and $\mu^{0}_{-,2}, \mu^{0}_{+,1}$ are analytic in $\mathbb{C}^{-}$. The $\mu^{0}_{\pm,j} (j=1,2)$ mean the $j$-th column of $\mu^{0}_{\pm}$, and $\mathbb{C}^{\pm}$ mean the upper and lower complex $k$-plane, respectively.
\end{prop}

Furthermore, in order to  study the asymptotic property of $\mu^{0}_{\pm}$ as $z\rightarrow0$
and  make the following analysis more clear, we introduce a transformation, i.e., $k=\frac{z}{2\sqrt{1+q_{0}^2}}=\frac{z}{2\Phi_{0}}$ which directly leads to $\lambda=\frac{z}{2}$.
For convenience, we will replace $\mu^{0}_{\pm}(x,t,k)$ with $\mu^{0}_{\pm}(x,t,z)$  in the following analysis.
Then, we expand the $Y_{\pm}^{-1}\mu^{0}_{\pm}$ as
\begin{align}\label{Ymu0-expand}
Y^{-1}_{\pm}\mu_{\pm}&=\chi_{\pm}^{(0)}+\chi_{\pm}^{(1)}z
+\ldots\notag \\
&=Y^{-1}_{\pm}\left(\mu_{\pm}^{0,(0)}+\frac{\mu_{\pm}^{0,(1)}}{z}+\ldots\right),\quad z\rightarrow0.
\end{align}
Substituting the above asymptotic expansions
into the Lax pair \eqref{mu0-Lax} and comparing the same power coefficients of $z$, we obtain the expressions of $\mu^{0,(0)}_{\pm}$ and $\mu^{0,(1)}_{\pm}$. It should be pointed out that $\mu^{0,(j)}_{\pm}(j=1,2,\ldots)$ are independent of $z$.
\begin{prop}
The functions $\mu^{0}_{\pm}(x,t;z)$ admit the following asymptotic property as $z\rightarrow0$,
\begin{align}\label{u0-asym}
\mu^{0}_{\pm}(x,t;z)=Y_{\pm}+\frac{z}{2\Phi_{0}}\int_{\pm\infty}^{x}
\left(
\begin{array}{cc}
\frac{i\bar{q}_{\pm}(q-q_{\pm}}{\Phi_{0}+1} & q-q_{\pm} \\
 -\bar{q}+\bar{q}_{\pm} & -\frac{iq_{\pm}(\bar{q}-\bar{q}_{\pm}}{\Phi_{0}+1} \\
  \end{array}
    \right)
\,dx+O(z^{2}).
\end{align}
\end{prop}

\subsection{The singularity at $k=\infty$}
To maintain consistency of analysis, we still use the transformation $k=\frac{z}{2\Phi_{0}}$ to
control the asymptotic behavior of the Lax pair \eqref{O-Lax} as $z\rightarrow\infty(k\rightarrow\infty)$.  Then, the Lax pair \eqref{O-Lax} is transformed into
\begin{gather}
\phi_{x}=X_{1}\phi,~~~ \phi_{t}=T_{1}\phi, \label{O1-Lax}
\end{gather}
where
\begin{gather}
X_{1}=-i\frac{1}{2\Phi_{0}}z\sigma_{3}+\frac{1}{2\Phi_{0}}zQ, \notag \\
T_{1}=\left(
    \begin{array}{cc}
      -\frac{1}{2\Phi^{2}_{0}}\frac{iz^{2}}{\Phi} & \frac{1}{2\Phi^{2}_{0}}\frac{qz^{2}}{\Phi}+iz\frac{1}{2\Phi_{0}}\left(\frac{q}{\Phi}\right)_{x} \\
      -\frac{1}{2\Phi^{2}_{0}}\frac{\bar{q}z^{2}}{\Phi}+iz\frac{1}{2\Phi_{0}}\left(\frac{\bar{q}}{\Phi}\right)_{x} & \frac{1}{2\Phi^{2}_{0}}\frac{iz^{2}}{\Phi} \\
    \end{array}
  \right)\notag
\end{gather}
Next, following the idea \cite{Fan-mCH-Dbar,Fan-SP-Dbar} , we introduce the transformation
\begin{align}\label{G-Trans}
\phi(x,t;z)=G(x,t)\psi(x,t;z),
\end{align}
where
\begin{align*}
  G(x,t)=\sqrt{\frac{\Phi+1}{2\Phi}}\left(
                                              \begin{array}{cc}
                                                1 & \frac{i(1-\Phi)}{\bar{q}(x,t)} \\
                                                \frac{i(1-\Phi)}{q(x,t)} & 1 \\
                                              \end{array}
                                            \right).
\end{align*}
As a result, the Lax pair related to $\phi(x,t;z)$ is derived as
\begin{align}\label{psi-Lax}
\begin{split}
&\psi_{x}+\frac{iz}{2\Phi_{0}}\Phi\sigma_{3}\psi=X_{2}\psi,\\
&\psi_{t}+\left(i\frac{z^{2}}{2\Phi^{2}_{0}}
+z\frac{q\bar{q}_{x}-q_x\bar{q}}{4\Phi^{2}}\Phi_{0}\right)\sigma_{3}\psi=T_{2}\psi,
\end{split}
\end{align}
where
\begin{align*}
X_2=\left(
      \begin{array}{cc}
        -\frac{q\bar{q}_{x}-q_x\bar{q}}{4\Phi(1+\Phi)} & -\frac{iq[\Phi(q\bar{q}_{x}-q_x\bar{q})-|q|_{x}^{2}}{4\Phi^{2}(\Phi^{2}-1)} \\
        \frac{i\bar{q}[\Phi(q\bar{q}_{x}-q_x\bar{q})+|q|_{x}^{2}}{4\Phi^{2}(\Phi^{2}-1)} & \frac{q\bar{q}_{x}-q_x\bar{q}}{4\Phi(1+\Phi)} \\
      \end{array}
    \right),~~T_2=\left(
                    \begin{array}{cc}
                      T_{2,11} & T_{2,12} \\
                      T_{2,21} & -T_{2,11} \\
                    \end{array}
                  \right)
\end{align*}
with
\begin{align*}
T_{2,11}&=-\frac{q\bar{q}_{t}-q_t\bar{q}}{4\Phi(1+\Phi)},\\
T_{2,12}&=\frac{iq[\bar{q}(q\bar{q}_{x}-q_x\bar{q})-2\bar{q}_{x}(\Phi-1)}{2\Phi^{3}(\Phi-1)\bar{q}}z- \frac{iq[\bar{q}(q\bar{q}_{t}-q_t\bar{q})-2\bar{q}_{t}(\Phi-1)}{4\Phi^{2}(\Phi-1)\bar{q}},\\
T_{2,21}&=\frac{i\bar{q}[q(q\bar{q}_{x}-q_x\bar{q})+2q_{x}(\Phi-1)}{2\Phi^{3}(\Phi-1)q}z+ \frac{i\bar{q}[q(q\bar{q}_{t}-q_t\bar{q})+2q_{t}(\Phi-1)}{4\Phi^{2}(\Phi-1)q}.
\end{align*}
In order to make the following analysis more clear, we
define $p_{x}(x,t;z)=\frac{iz}{2\Phi_{0}}\Phi$ and $p_{t}(x,t;z)=(\frac{iz^{2}}{2\Phi^{2}_{0}}+z\frac{q\bar{q}_{x}-q_x\bar{q}}{4\Phi^{2}_{0}\Phi^{2}})$. Then $p_{x}$ and $p_{t}$ are compatible, i.e., $p_{xt}=p_{tx}$. We rewrite this relation as
\begin{align*}
i\Phi_{t}=\left(\frac{q\bar{q}_{x}-q_x\bar{q}}{2\Phi^{2}}\right)_{x},
\end{align*}
which is the conservation law \cite{WKI-conservation-law} of the WKI equation.  Therefore, $p(x,t;z)$ is  defined as
\begin{align}\label{define-p}
p(x,t;z)=i\frac{z}{2}\left(x-\int_{x}^{-\infty}(\frac{\Phi(y)}{\Phi_{0}}-1)\,dy\right)+\frac{iz^{2}}{2\Phi_{0}^{2}}t.
\end{align}
Furthermore, by defining $\varphi=\psi e^{p(x,t;z)\sigma_{3}}$, we can directly derive the equivalent  Lax pair
\begin{align*}
\varphi_x(x,t;z)+p_{x}(x,t;z)[\sigma_3, \varphi(x,t;z)]=X_2\varphi(x,t;z),\\
\varphi_t(x,t;z)+p_{t}(x,t;z)[\sigma_3, \varphi(x,t;z)]=T_2\varphi(x,t;z).
\end{align*}
Considering the fact that the potential $q(x,t)$ is complex valued, the diagonal elements of the matrix $X_2$ do not equal to zero, which gives rise to the solutions of spectral problem do not approximate the identity matrix as $z\rightarrow\infty$.  To overcome this difficult,  we make a further transformation as
\begin{align}\label{ImpG-Trans}
\varphi(x,t;z)=e^{d_{+}\hat{\sigma}_{3}}\mu(x,t;z)e^{-d_{-}\sigma_{3}},
\end{align}
where
\begin{gather}
d_{-}=\int_{-\infty}^{x}\frac{q\bar{q}_{x}-q_{x}\bar{q}}{4\Phi(\Phi+1)}(s,t)ds,
~~d_{+}=\int^{+\infty}_{x}\frac{q\bar{q}_{x}-q_{x}\bar{q}}{4\Phi(\Phi+1)}(s,t)ds,\notag\\
d=d_{+}+d_{-}=\int_{-\infty}^{+\infty} \frac{q\bar{q}_{x}-q_{x}\bar{q}}{4\Phi(\Phi+1)}(s,t)ds.\label{defin-d}
\end{gather}
Then, based on the Lax pair related to $\varphi(x,t;z)$, we derive a equivalent Lax pair  as
\begin{align}\label{Lax-mu}
\begin{split}
\mu_{x}+p_{x}[\sigma_{3},\mu]=-e^{-d_{+}\hat{\sigma}_{3}}X_{3}\mu,\\
\mu_{t}+p_{t}[\sigma_{3},\mu]=-e^{-d_{+}\hat{\sigma}_{3}}T_{3}\mu,
\end{split}
\end{align}
where
\begin{align*}
X_3=X_2+\frac{q\bar{q}_{x}-q_{x}\bar{q}}{4\Phi(\Phi+1)}\sigma_3,\\
T_3=T_2+\frac{q\bar{q}_{t}-q_{t}\bar{q}}{4\Phi(\Phi+1)}\sigma_3.
\end{align*}
Similar to the case of $z=0$, \eqref{Lax-mu} can be written as the full differential form
\begin{align}\label{fullDerivative-2}
d(e^{-p(x,t;z)\hat{\sigma}_{3}}\mu)=e^{-p(x,t;z)\hat{\sigma}_{3}} e^{-d_{+}\hat{\sigma}_{3}}(X_{3}dx+T_{3}dt)\mu,
\end{align}
from which, the Volterra type integral equations are derived as
\begin{align}\label{Volterra-2}
\mu_{-}(x,t;z)=\mathbb{I}+
\int_{x}^{\pm\infty}e^{[p(x,t;z)-p(s,t;z)]\hat{\sigma}_{3}}e^{-d_{+}\hat{\sigma}_{3}}
X_{3}(s,t;z)\mu_{\pm}(s,t;z)\,ds,
\end{align}
where $\mu_{\pm}=\mu_{\pm}(x,t;z)$ mean the two eigenfunction solutions of the Lax pair \eqref{Lax-mu}.
Then, referring to the definition of $\mu(x,t;z)$ and the  integrals \eqref{Volterra-2}, the properties of the eigenfunctions $\mu_{\pm}(x,t;z)$ can be derived directly.
\begin{prop}\label{Anal-mu}
(Analytic property) It is assumed that $q(x)-q_{+}\in L^{1}(a,+\infty)$ and $q(x)-q_{-}\in L^{1}(-\infty,a)$ hold for any constant $a\in \mathbb{R}$. Then, $\mu_{-,1}, \mu_{+,2}$ are analytic in $\mathbb{C}^{+}$ and $\mu_{-,2}, \mu_{+,1}$ are analytic in $\mathbb{C}^{-}$. The $\mu_{\pm,j} (j=1,2)$ mean the $j$-th column of $\mu_{\pm}$.
\end{prop}

Furthermore, by simple calculations, similar to the reference \cite{Li-CSP-Dbar},  the symmetry and asymptotic properties of  $\mu_{\pm}$ can be derived.
\begin{prop}\label{Sym-mu}
(Symmetry property) The eigenfunctions $\mu_{\pm}(x,t;z)$ satisfy the following symmetry relation
\begin{align}
    \bar{\mu}_{\pm}(x,t;\bar{z})=-\sigma_{2}\mu_{\pm}(x,t;z)\sigma_{2}.
  \end{align}
\end{prop}

\begin{prop}\label{Asy-mu}
(Asymptotic property for $z\rightarrow\infty$) The eigenfunctions $\mu_{\pm}(x,t;z)$ satisfy the following asymptotic behavior
      \begin{align}
       \mu_{\pm}(x,t;z)=\mathbb{I}+O(z^{-1}),~~ z\rightarrow\infty.
      \end{align}
\end{prop}

\subsection{The scattering matrix}

Because both eigenfunctions $\mu_{+}(x,t;z)$ and $\mu_{-}(x,t;z)$ are the fundamental matrix solutions of Eq.\eqref{Lax-mu} , there exists a linear relation between $\mu_{+}(x,t;z)$ and $\mu_{-}(x,t;z)$, i.e.,
\begin{align}\label{Sz}
\mu_{+}(x,t;z)=\mu_{-}(x,t;z)e^{-p(x,t;z)\hat{\sigma}_{3}}S(z),~~z\in \mathbb{R},
\end{align}
where $S(z)=(s_{ij}(z))~(i,j=1,2)$ is named the scattering matrix which is independent of the variable $x$ and $t$. The coefficients $s_{11}(z)$ and $s_{22}(z)$ can be expressed as
\begin{align}\label{s-mu}
s_{11}(z)=\det(\mu_{+,1}, \mu_{-,2}),~~s_{22}(z)=\det(\mu_{-,1}, \mu_{+,2}).
\end{align}
Then, on the basis of the above propositions and the definition of the scattering  matrix $S(z)$, we obtain the following standard results. The similar proofs can be found in many literatures[see e.g.\cite{Deift-2003}].

\begin{prop}\label{Prop-S}
The scattering matrix $S(z)$ possesses the properties:
\begin{itemize}
  \item (Analytic property) The scattering coefficients $s_{11}$ and $s_{22}$ are respectively analytic in $\mathbb{C}^{-}$ and $\mathbb{C}^{+}$.
  \item (Symmetry property) The scattering coefficients $s_{ij}(i,j=1,2)$ possess the following relations:
  \begin{align}\label{sym-S}
    s_{11}(z)=\overline{s_{22}(z)},~~s_{12}(z)=-\overline{s_{21}(z)}.
  \end{align}
  \item (Asymptotic property for $z\rightarrow\infty$)
  \begin{align}
       S(z)=\mathbb{I}+O(z^{-1}),~~ z\rightarrow\infty.
      \end{align}
\end{itemize}
\end{prop}

Additionally, we define the reflection coefficient as
\begin{gather}\label{r-expression}
r(z)=\frac{s_{12}(z)}{s_{22}(z)},
\end{gather}
then, according to \eqref{sym-S}, the reflection coefficient satisfies that  $$\frac{s_{12}(z)}{s_{22}(z)}=-\frac{\overline{s_{21}}(\bar{z})}{\overline{s_{11}}(\bar{z})}
 =-\bar{r}(\bar{z})=-\bar{r}(z)$$ for $z\in\mathbb{R}$.

 Moreover, the trace formula of the scattering coefficient $s_{22}(z)$ can be derived as
\begin{align}\label{s22-trace}
s_{22}(z)=\prod_{k=1}^{N}\frac{z-z_{k}} {z-\bar{z}_{k}}\exp\left(i\int_{-\infty}^{+\infty}\frac{\nu(s)}{s-z}ds\right).
\end{align}

\subsection{The connection between $\mu_{\pm}(x,t;z)$ and $\mu^{0}_{\pm}(x,t;z)$}

In the following part,  the eigenfunctions $\mu_{\pm}(x,t;z)$ will be used  to construct the matrix $M(x,t;z)$. Then, a RHP can be  formulated . While in order to obtain the reconstruction formula between the solution $q(x,t)$ and the formulated RHP, the asymptotic behavior of $\mu_{\pm}$ as $z\rightarrow0$ is necessary. Thus, we need to establish the connection between $\mu_{\pm}(x,t;z)$ and $\mu^{0}_{\pm}(x,t;z)$.

According to the transformation \eqref{mu0}, \eqref{G-Trans} and \eqref{ImpG-Trans} ,  we assume that the eigenfunctions $\mu_{\pm}(x,t;z)$ and $\mu^{0}_{\pm}(x,t;z)$ related to each other as
\begin{align}\label{mu-mu0}
\mu_{\pm}(x,t;z)=e^{-d_{+}\sigma_{3}}G^{-1}(x,t)\mu^{0}_{\pm}(x,t;z)
e^{-i\lambda(x+\frac{2k}{\Phi_{0}}t)\sigma_{3}}C_{\pm}(z)e^{p(x,t;z)\sigma_{3}}e^{d\sigma_{3}},
\end{align}
where $C_{\pm}(z)$ are independent of $x$ and $t$.
Taking $x\rightarrow\infty$, from \eqref{mu-mu0} , we can derive that
\begin{align*}
C_{-}(z)=\sqrt{\frac{2}{\Phi_{0}(1+\Phi_{0})}},~~
C_{+}(z)=\sqrt{\frac{2}{\Phi_{0}(1+\Phi_{0})}}e^{-d\sigma_{3}}e^{-izc\sigma_{3}},
\end{align*}
where $c=\int^{+\infty}_{-\infty}(\frac{\Phi(s)}{\Phi_0}-1)ds$ is a quantity conserved under the dynamics governed by \eqref{WKI-equation}. As a result, we obtain
\begin{align}\label{mu-pm-mu0}
\begin{split}
\mu_{-}(x,t;z)=\sqrt{\frac{2}{\Phi_{0}(1+\Phi_{0})}}e^{-d_{+}\sigma_{3}}G^{-1}(x,t)\mu^{0}_{-}(x,t;z)
e^{-iz\int^{x}_{-\infty}(\frac{\Phi(s)}{\Phi_0}-1)ds\sigma_{3}}e^{d\sigma_{3}},\\
\mu_{+}(x,t;z)=\sqrt{\frac{2}{\Phi_{0}(1+\Phi_{0})}}e^{-d_{+}\sigma_{3}}G^{-1}(x,t)\mu^{0}_{+}(x,t;z)
e^{-iz\int^{+\infty}_{x}(\frac{\Phi(s)}{\Phi_0}-1)ds\sigma_{3}}.
\end{split}
\end{align}

\section{The Riemann-Hilbert problem for WKI equation}

In order to avoid many possible pathologies in the following part, we first make some assumptions.
\begin{assum}\label{assum}
For the Cauchy problem of WKI equation \eqref{WKI-equation}, let $q\mp q_{\pm}\in L^{1,2}(\mathbb{R})$, the initial value generates generic scattering data in the sense that:
\begin{itemize}
  \item For $z\in\mathbb{R}$, no spectral singularities exist, i.e., $s_{22}(z)\neq0$ $(z\in\mathbb{R})$;
  \item Suppose that $s_{22}(z)$ possesses $N$ zero points, denoted as $\mathcal{Z}=\left\{(z_{j},Im~z_{j}>0)^{N}_{j=1}\right\}$.
  \item The discrete spectrum is simple, i.e., if $z_{0}$ is the zero of $s_{22}(z)$, then $s'_{22}(z_{0})\neq0$.
\end{itemize}
\end{assum}

Define weighted Sobolev spaces
\begin{align}\label{W-H-Sobolev-spaces}
\begin{split}
&W^{k,p}(\mathbb{R})=\left\{f(x)\in L^{p}(\mathbb{R}):\partial^{j}f(x)\in L^{p}(\mathbb{R}), j=1,2,\ldots,k\right\},\\
&H^{k,2}(\mathbb{R})=\left\{f(x)\in L^{p}(\mathbb{R}):x^{2}\partial^{j}f(x)\in L^{p}(\mathbb{R}), j=1,2,\ldots,k\right\},\\
&\mathcal{H}(\mathbb{R})=W^{1,1}(\mathbb{R})\cap H^{2,2}(\mathbb{R}),
\end{split}
\end{align}
then we can further show that
\begin{prop}\label{prop-r-map}
For any given $q(x,t)\mp q_{\pm}\in L^{1,2}$, $q'(x,t)\in\mathcal{H}(R)$, then $r(z)\in H^{1}(\mathbb{R})$.
\end{prop}
\begin{proof}
The proof is similar to \cite{Cuccagna-2016}, so we omit it.
\end{proof}

Next, we define a sectionally  meromorphic matrices
\begin{align}\label{Matrix}
\tilde{M}(x,t;z)=\left\{\begin{aligned}
&\tilde{M}^{+}(x,t;z)=\left(\mu_{-,1}(x,t;z),\frac{\mu_{+,2}(x,t;z)}{s_{22}(z)}\right), \quad z\in \mathbb{C}^{+},\\
&\tilde{M}^{-}(x,t;z)=\left(\frac{\mu_{+,1}(x,t;z)}{s_{11}(z)},\mu_{-,2}(x,t;z)\right), \quad z\in \mathbb{C}^{-},
\end{aligned}\right.
\end{align}
where $\tilde{M}^{\pm}(x,t;z)=\lim\limits_{\varepsilon\rightarrow0^{+}}\tilde{M}(x,t;z\pm i\varepsilon),~\varepsilon\in\mathbb{R}$.

For the initial data that admits the Assumption \ref{assum}, the matrix function $\tilde{M}(x,t;z)$ solves the following matrix RHP.

\begin{RHP}\label{RH-1}
Find an analytic function $\tilde{M}(x,t;z)$ with the following properties:
\begin{itemize}
  \item $\tilde{M}(x,t;z)$ is meromorphic in $\mathbb{C}\setminus\mathbb{R}$;
  \item $\tilde{M}^{+}(x,t;z)=\tilde{M}^{-}(x,t;z)\tilde{V}(x,t;z)$,~~~$z\in\mathbb{R}$,
  where \begin{align}\label{J-Matrix-1}
\tilde{V}(x,t;z)=\left(\begin{array}{cc}
                   1 & r(z)e^{-2zp} \\
                   -\bar{r}(z)e^{2zp} & 1+|r(z)|^{2}
                 \end{array}\right);
\end{align}
  \item $\tilde{M}(x,t;z)=\mathbb{I}+O(z^{-1})$ as $z\rightarrow\infty$.
\end{itemize}
\end{RHP}
According  to \eqref{s-mu} and the Assumption \ref{assum},  there exist norming constants $b_{j}$ such that
\begin{equation}
\mu_{+,2}(z_{j})=b_{j}e^{-2p(z_{j})}\mu_{-,1}(z_{j}).\notag
\end{equation}
As a result, the residue condition of $\tilde{M}(x,t;z)$ can be expressed as
\begin{align}\label{res-N}
\mathop{Res}_{z=z_{j}}\tilde{M}=\lim_{z\rightarrow z_{j}}\tilde{M}\left(\begin{array}{cc}
                   0 & c_{j}e^{-2p(z_{j})} \\
                   0 & 0
                 \end{array}\right),~~
\mathop{Res}_{z=\bar{z}_{j}}\tilde{M}=\lim_{z\rightarrow \bar{z}_{j}}\tilde{M}\left(\begin{array}{cc}
                   0 & 0 \\
                   -\bar{c}_{j}e^{2p(\bar{z})} & 0
                 \end{array}\right),
\end{align}
where $c_{j}=\frac{b_{j}}{s'_{22}(z_{j})}$.
\begin{rem}
On the basis of the Zhou's vanishing lemma \cite{Zhou-vanishing-lemma-CPAM}, the existence of the solutions of RHP \ref{RH-1} for $(x,t)\in\mathbb{R}^{2}$ is guaranteed. According to the results of Liouville's theorem, we know that if the solution of RHP \ref{RH-1} exists, it is unique.
\end{rem}

Next, in order  to reconstruct the solution $q(x,t)$ of the WKI equation \eqref{WKI-equation}, we need to study the asymptotic behavior of $\tilde{M}(x,t;z)$ as $z\rightarrow 0$, i.e.,
\begin{align}\label{N--0}
\tilde{M}(x,t;z)=e^{-d_{+}\sigma_{3}}G^{-1}(x,t)\left[\mathbb{I}+\frac{z}{2\Phi_{0}}\left(
\int_{\pm\infty}^{x}\left(
  \begin{array}{cc}
    0 & q-q_{\pm} \\
    -\bar{q}+\bar{q}_{\pm} & 0 \\
  \end{array}
\right)\,dx-iY_{\pm}c_{-}\sigma_{3}\right)+O(z^{2})\right]e^{d\sigma_{3}},~~z\rightarrow0,
\end{align}
where $c_{-}(x,t)=\int^{-\infty}_{x}(\Phi(s,t)-1)ds$.
However, because $p(x,t;z)$ that appears in jump matrix \eqref{J-Matrix-1} is not clear, it is quite difficult to reconstruct the solution $q(x,t)$ from \eqref{N--0}. Boutet de Monvel and Shepelsky have overcome this problem by changing the spatial variable of the Camassa-Holm equation and short wave equations \cite{Boutet-Shepelsky-1,Boutet-Shepelsky-2}. Thus, following the idea, we introduce a new scale
\begin{align}\label{y-scale}
y(x,t)=x-\int^{-\infty}_{x}(\frac{\Phi(s,t)}{\Phi_{0}}-1)ds=x-c_{-}(x,t).
\end{align}
With this new scale, the jump matrix can be expressed explicitly. Meanwhile, the new scale will give rise to that the solution $q(x,t)$ can be expressed only in implicit form: It will be given in terms of functions in the new scale, whereas the original scale will also be given in terms of functions in the new scale.
According to the definition of $y(x,t)$, we define
\begin{align*}
       \tilde{M}(x,t;z)=M(y(x,t),t;z),
\end{align*}
then, the $M(y(x,t),t;z)$ satisfies the following matrix RHP.
\begin{RHP}\label{RH-2}
Find an analytic function $M(y,t;z)$ with the following properties:
\begin{itemize}
  \item $M(y,t;z)$ is meromorphic in $\mathbb{C}\setminus\mathbb{R}$;
  \item $M^{+}(y,t;z)=M^{-}(y,t;z)V(y,t;z)$,~~~$z\in\mathbb{R}$,
  where \begin{align}\label{J-Matrix}
V(y,t;z)=e^{-i\left(\frac{z}{2}y+\frac{z^{2}}{2\Phi_{0}}t\right)\hat{\sigma}_{3}}\left(\begin{array}{cc}
                   1 & r(z) \\
                   \bar{r}(z) & 1+|r(z)|^{2}
                 \end{array}\right);
\end{align}
  \item $M(y,t;z)=\mathbb{I}+O(z^{-1})$ as $z\rightarrow\infty$;
  \item $M(y,t;z)$ possesses simple poles at each point in $\mathcal{Z}\cup\bar{\mathcal{Z}}$ with:
     \begin{align}\label{res-M}
     \begin{split}
\mathop{Res}_{z=z_{j}}M(z)=\lim_{z\rightarrow z_{j}}M(z)\left(\begin{array}{cc}
                   0 & c_{j}e^{-i(z_{j}y+\frac{z^{2}_{j}}{\Phi_{0}}t)} \\
                   0 & 0
                 \end{array}\right),\\
\mathop{Res}_{z=\bar{z}_{j}}M(z)=\lim_{z\rightarrow \bar{z}_{j}}M(z)\left(\begin{array}{cc}
                   0 & 0 \\
                   -\bar{c}_{j}e^{i(\bar{z}_{j}y+\frac{\bar{z}^{2}_{j}}{\Phi_{0}}t)} & 0
                 \end{array}\right).
      \end{split}
\end{align}
\end{itemize}
\end{RHP}
\begin{prop}
If $M(y,t;z)$ satisfies the above conditions, then the RHP \ref{RH-2} possesses a unique solution. Additionally,  based on the solution of RHP \ref{RH-2}, the solution $q(x,t)$  of the initial value problem \eqref{WKI-equation} can be derived in parametric form, i.e., $q(x,t)=q(y(x,t),t)$ where
\begin{align}\label{q-sol}
q(y,t)=q_{\pm}+ \frac{(\Phi_{0}+1)^{4}}{(\Phi_{0}+1)^{4}-q_{0}^{4}}\mathcal{M}(x,t),\\
c_{-}=\lim_{z\rightarrow0}\frac{1}{iz}((M(y,t,0)^{-1}M(y,t,z))_{11}-1) -\frac{1}{\Phi_{0}+1}(\bar{q}_{\pm}\mathcal{M}(x,t)+q_{\pm}\bar{\mathcal{M}}(x,t))
\end{align}
where
\begin{align*}
\mathcal{M}(x,t)=\left(\lim_{z\rightarrow0}\frac{\partial}{\partial y}\frac{1}{z}(M(y,t,0)^{-1}M(y,t,z))_{12}e^{2d} \frac{4\Phi_{0}^{2}}{(\Phi_{0}+1)}\right.\\
 \left.-\frac{q_{\pm}^{2}}{\Phi_{0}+1}\lim_{z\rightarrow0}\frac{\partial}{\partial y}\frac{1}{z}(M(y,t,0)^{-1}M(y,t,z))_{21}e^{-2d} \frac{4\Phi_{0}^{2}}{(\Phi_{0}+1)}\right).
\end{align*}
\end{prop}
\begin{proof}
A fact that the jump matrix $V(y,t;z)$ is a Hermitian matrix gives rise to that the RHP \ref{RH-2} indeed has a  solution. Moreover, due to the normalize condition, i.e., $M(y,t;z)=\mathbb{I}+O(z^{-1})$ as $z\rightarrow\infty$, the RHP \ref{RH-2} possesses only one solution.

Referring to the asymptotic formula \eqref{N--0}, the expression of the solution $q(x,t)$ can be derived.
\end{proof}

\section{Interpolation and conjugation}\label{section-Conjugation}

In this section, we are going to give two important operations to apply the $\bar{\partial}$ steepest method to analyse the RHP as $t\rightarrow \infty$: Interpolate the poles through converting them to
jumps on  small contours encircling each pole; The factorization of the jump matrix
along the real axis is used to deform the contours onto those on which the oscillatory jump
on the real axis convert to exponential decay.

The oscillation term in jump matrix \eqref{J-Matrix} can be rewritten as $e^{-it\theta(z)}\left(\theta(z)=\left(\frac{zy}{t}+\frac{z^{2}}{\Phi_{0}}\right)\right)$ from which we obtain a phase point
\begin{align*}
z_{0}=-\left(\frac{y}{2t}\Phi_{0}\right).
\end{align*}
So, the real part of $2it\theta(z)$ is expressed as
\begin{align}\label{Reitheta}
Re (i\theta)=-\frac{2}{\Phi_{0}}(Rez-z_{0})Im z,
\end{align}
from which we can  derive the decaying domains of the oscillation term, see Figure 1.

\centerline{\begin{tikzpicture}[scale=0.7]
\path [fill=yellow] (-5,4)--(-5,0) to (-9,0) -- (-9,4);
\path [fill=yellow] (-5,-4)--(-5,0) to (-1,0) -- (-1,-4);
\draw[-][thick](-9,0)--(-8,0);
\draw[-][thick](-8,0)--(-7,0);
\draw[-][thick](-7,0)--(-6,0);
\draw[-][thick](-6,0)--(-5,0);
\draw[-][thick](-5,0)--(-3,0);
\draw[-][thick](-4,0)--(-3,0);
\draw[-][thick](-3,0)--(-2,0);
\draw[->][thick](-2,0)--(-1,0)[thick]node[right]{$Rez$};
\draw[<-][thick](-5,4)--(-5,3);
\draw[fill] (-5,3.5)node[right]{$Imz$};
\draw[fill] (-5,4)node[above]{$t\rightarrow-\infty$};
\draw[-][thick](-5,3)--(-5,2);
\draw[-][thick](-5,2)--(-5,1);
\draw[-][thick](-5,1)--(-5,0);
\draw[-][thick](-5,0)--(-5,1);
\draw[-][thick](-5,1)--(-5,2);
\draw[-][thick](-5,2)--(-5,3);
\draw[-][thick](-5,3)--(-5,4);
\draw[-][thick](-5,0)--(-5,-1);
\draw[-][thick](-5,-1)--(-5,-2);
\draw[-][thick](-5,-2)--(-5,-3);
\draw[-][thick](-5,-3)--(-5,-4);
\draw[fill] (-3,1.5)node[below]{\small{$|e^{-i\theta(z)}|\rightarrow0$}};
\draw[fill] (-7.5,-1.5)node[below]{\small{$|e^{-i\theta(z)}|\rightarrow0$}};
\draw[fill] (-3,-1.5)node[below]{\small{$|e^{i\theta(z)}|\rightarrow0$}};
\draw[fill] (-7.5,1.5)node[below]{\small{$|e^{i\theta(z)}|\rightarrow0$}};
\path [fill=yellow] (5,4)--(5,0) to (9,0) -- (9,4);
\path [fill=yellow] (5,-4)--(5,0) to (1,0) -- (1,-4);
\draw[-][thick](1,0)--(8,0);
\draw[->][thick](8,0)--(9,0)[thick]node[right]{$Rez$};
\draw[<-][thick](5,4)--(5,3);
\draw[fill] (5,3.5)node[right]{$Imz$};
\draw[fill] (5,4)node[above]{$t\rightarrow+\infty$};
\draw[-][thick](5,0)--(5,4);
\draw[-][thick](5,0)--(5,-1);
\draw[-][thick](5,-1)--(5,-2);
\draw[-][thick](5,-2)--(5,-3);
\draw[-][thick](5,-3)--(5,-4);
\draw[fill] (7,1.5)node[below]{\small{$|e^{i\theta(z)}|\rightarrow0$}};
\draw[fill] (2.5,-1.5)node[below]{\small{$|e^{i\theta(z)}|\rightarrow0$}};
\draw[fill] (7,-1.5)node[below]{\small{$|e^{-i\theta(z)}|\rightarrow0$}};
\draw[fill] (2.5,1.5)node[below]{\small{$|e^{-i\theta(z)}|\rightarrow0$}};
\end{tikzpicture}}
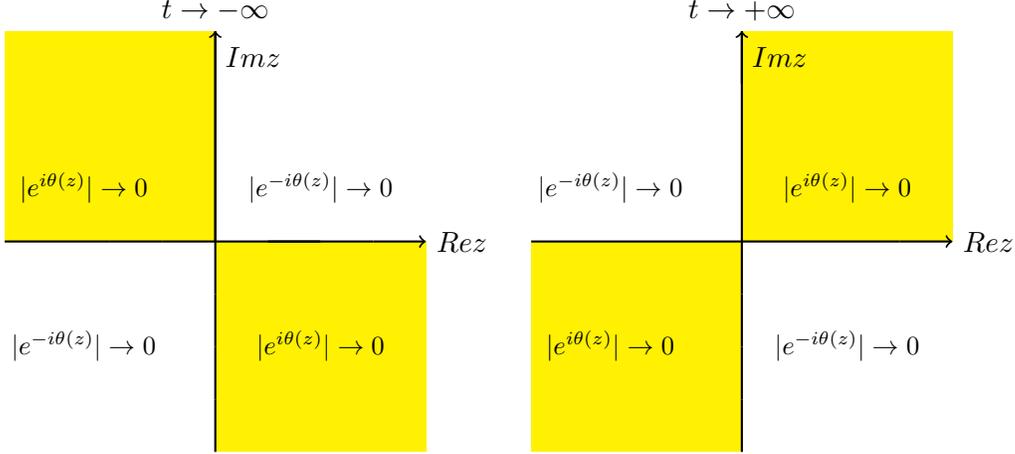
\centerline{\noindent {\small \textbf{Figure 1.} Exponential decaying domains.}}

From the exponential decaying domains shown in Figure 1, we know that  the two situations are exactly opposite as $t\rightarrow\infty$ and $t\rightarrow-\infty$. Therefore,
in the following part, we mainly focus on the case that $t\rightarrow-\infty$, and the analysis of the case $t\rightarrow+\infty$ is essentially the same.

To make the following analyses more clear, we introduce some notations.
\begin{align}\label{notation-1}
\begin{aligned}
&\triangle^{-}_{z_{0}}=\{k\in\{1,\cdots,N\}|Re(z_{k})<z_{0}\},\\
&\triangle^{+}_{z_{0}}=\{k\in\{1,\cdots,N\}|Re(z_{k})>z_{0}\}.
\end{aligned}
\end{align}
For $\mathcal{I}=[a,b]$, define
\begin{align*}
&\mathcal{Z}(\mathcal{I})=\{z_{k}\in\mathcal{Z}:Rez_{k}\in\mathcal{I}\},\\
&\mathcal{Z}^{-}(\mathcal{I})=\{z_{k}\in\mathcal{Z}:Rez_{k}<a\},\\
&\mathcal{Z}^{+}(\mathcal{I})=\{z_{k}\in\mathcal{Z}:Rez_{k}>b\}.
\end{align*}

Next, in order to re-normalize the Riemann-Hilbert problem such that it is well behaved for $t\rightarrow-\infty$ with fixed phrase point $z_0$, we first introduce  the function
\begin{align}\label{delta-v-define}
\delta(z)=\exp\left[i\int_{-\infty}^{z_{0}}\frac{\nu(s)}{s-z}ds\right],~~\nu(s)=-\frac{1}{2\pi}\log(1+|r(s)|^{2}),
\end{align}
and
\begin{align}\label{T-define}
T(z)=T(z,z_{0})=\prod_{k\in\Delta_{z_{0}}^{+}}\frac{z-\bar{z}_{k}}{z-z_{k}}\delta(z).
\end{align}

\begin{prop}\label{T-property} The function $T(z,z_{0})$ satisfies that\\
($a$) $T$ is meromorphic in $C\setminus(-\infty, z_{0}]$. $T(z,z_{0})$ possesses simple pole at $z_{k}(k\in\triangle^{+}_{z_{0}})$ and simple zero at $\bar{z}_{k}(k\in\triangle^{+}_{z_{0}})$.\\
($b$) For $z\in C\setminus(-\infty,z_{0}]$, $\bar{T}(\bar{z})=\frac{1}{T(z)}$.\\
($c$) For $z\in (-\infty,z_{0}]$, the boundary values $T_{\pm}$ satisfy that
\begin{align}\label{3-3}
\frac{T_{+}(z)}{T_{-}(z)}=1+|r(z)|^{2}, z\in (-\infty,z_{0}).
\end{align}
($d$) As $|z|\rightarrow \infty $ with $|arg(z)|\leq c<\pi$,
\begin{align}\label{3-4}
T(z)=1+\frac{i}{z}\left[2\sum_{k\in\Delta_{z_{0}}^{+}}Imz_{k}-\int_{-\infty}^{z_{0}}\nu(s)ds\right]+O(z^{-2}).
\end{align}
($e$) As $z\rightarrow z_{0}$ along any ray $z_{0}+e^{i\phi}R_{+}$ with $|\phi|\leq c<\pi$
\begin{align}\label{3-5}
|T(z,z_{0})-T_{0}(z_{0})(z-z_{0})^{i\nu(z_{0})}|\leq C\parallel r\parallel_{H^{1}(R)}|z-z_{0}|^{\frac{1}{2}},
\end{align}
where $T_{0}(z_{0})$ is the complex unit
\begin{align}\label{3-6}
\begin{split}
&T_{0}(z_{0})=\prod_{k\in\Delta_{z_{0}}^{+}}(\frac{z_{0}-\bar{z}_{k}}{z_{0}-z_{k}})e^{i\beta(z_{0},z_{0})},\\
&\beta(z,z_{0})=-\nu(z_{0})\log(z-z_{0}+1)+\int_{-\infty}^{z_{0}}\frac{\nu(s)-\chi(s)\nu(z_{0})}{s-z}ds,
\end{split}
\end{align}
with $\chi(s)=1$ as $s\in(z_{0}-1, z_{0})$, and $\chi(s)=0$ as  $s\in(-\infty, z_{0}-1]$.\\
($f$) As $z\rightarrow0$, $T(z)$ can be expressed as
\begin{align}\label{4.10}
T(z)=T(0)(1+zT_{1})+O(z^{2}),
\end{align}
where $T_{1}=2\mathop{\sum}\limits_{k\in\triangle_{z_{0},1}^{+}}\frac{Im~z_{k}}{z_{k}} -\int_{-\infty}^{z_{0}}\frac{\nu(s)}{s^{2}}ds$.
\end{prop}

\begin{proof}
The above properties of $T(z)$ can be proved by a direct calculation, for details, see \cite{AIHP,Fan-SP-Dbar,Li-cgNLS}.
\end{proof}

Furthermore, we define
\begin{align}\label{rho-I-definition}
\rho=\frac{1}{2}\min\{\min_{i\neq j}|z_{i}-z{j}|, \min_{z_{j}\in\mathcal{Z}}\{Imz_{j}\}\},~~
I=(z_{0}-\rho,z_{0}+\rho).
\end{align}

Next, we are going to using the following transformation to trade the poles for jumps on small contours encircling each pole
\begin{align}\label{Trans-1}
M^{(1)}(z)=\left\{\begin{aligned}
      M(z)\left(
        \begin{array}{cc}
      1 & -\frac{c_{j}e^{-i\theta(z_{j})}}{z-z_{j}} \\
      0 & 1 \\
        \end{array}
      \right)T(z)^{\sigma_{3}}, ~~|z-z_{j}|<\rho,~~z_{j}\in\mathcal{Z}^{-}(I),\\
   M(z)\left(
        \begin{array}{cc}
      1 & 0 \\
      -\frac{z-z_{j}}{c_{j}e^{-i\theta(z_{j})}} & 1 \\
        \end{array}
      \right)T(z)^{\sigma_{3}}, ~~|z-z_{j}|<\rho,~~z_{j}\in\mathcal{Z}^{+}(I),\\
    M(z)\left(
        \begin{array}{cc}
      1 & 0 \\
      -\frac{\bar{c}_{j}e^{i\theta(\bar{z}_{j})}}{z-\bar{z}_{j}} & 1 \\
        \end{array}
      \right)T(z)^{\sigma_{3}}, ~~|z-z_{j}|<\rho,~~\bar{z}_{j}\in\mathcal{Z}^{-}(I),\\
      M(z)\left(
        \begin{array}{cc}
      1 &  -\frac{z-\bar{z}_{j}}{\bar{c}_{j}e^{i\theta(\bar{z}_{j})}} \\
     0 & 1 \\
        \end{array}
      \right)T(z)^{\sigma_{3}}, ~~|z-\bar{z}_{j}|<\rho,~~\bar{z}_{j}\in\mathcal{Z}^{+}(I),\\
       M(z)T(z)^{\sigma_{3}},~~~~~~~~~~~~~~~~~~~~~~~~~elsewhere.
   \end{aligned}\right.
\end{align}
Define contour
\begin{align}\label{Sigma1}
\Sigma^{(1)}=\mathbb{R}\cup\left(\bigcup_{z_{j}\in\mathcal{Z}^{\pm}(I)}\{z\in\mathbb{C}:~ |z-z_{j}|=\rho ~or ~|z-\bar{z}_{j}|=\rho\}\right).
\end{align}

\centerline{\begin{tikzpicture}[scale=0.7]
\draw[->][thick](-8,0)--(8,0)node[right]{$Rez$};
\draw[fill] (0,0.2)node[below]{$z_{0}$};
\draw[fill] (1,0)node[below]{$\rho$};
\draw[fill] (-1,0)node[below]{$-\rho$};
\draw[-][dashed](-1,3)--(-1,-3);
\draw[-][dashed](1,3)--(1,-3);
\draw[fill] (0.5,2.5)node[below]{$z_{1}$} circle [radius=0.08];
\draw[fill] (0.5,-2.5)node[below]{$\bar{z}_{1}$} circle [radius=0.08];
\draw[fill] (-0.5,1)node[below]{$z_{2}$} circle [radius=0.08];
\draw[fill] (-0.5,-1)node[below]{$\bar{z}_{2}$} circle [radius=0.08];
\draw[fill] (1.7,2.2)node[right]{$z_{3}$} ;
\draw[fill] (1.7,-2.2)node[right]{$\bar{z}_{3}$};
\draw[fill] (-5.3,1.5)node[left]{$z_{4}$};
\draw[fill] (-5.3,-1.5)node[left]{$\bar{z}_{4}$};
\draw[fill] (3.7,1.2)node[right]{$z_{5}$};
\draw[fill] (3.7,-1.2)node[right]{$\bar{z}_{5}$};
\draw[fill] (-3.3,1.2)node[left]{$z_{6}$};
\draw[fill] (-3.3,-1.2)node[left]{$\bar{z}_{6}$};
\draw(1.4,2.2) [black, line width=0.5] circle(0.3);
\draw(1.4,-2.2) [black, line width=0.5] circle(0.3);
\draw(-5,1.5) [black, line width=0.5] circle(0.3);
\draw(-5,-1.5) [black, line width=0.5] circle(0.3);
\draw(3.4,1.2) [black, line width=0.5] circle(0.3);
\draw(3.4,-1.2) [black, line width=0.5] circle(0.3);
\draw(-3,1.2) [black, line width=0.5] circle(0.3);
\draw(-3,-1.2) [black, line width=0.5] circle(0.3);
\end{tikzpicture}}
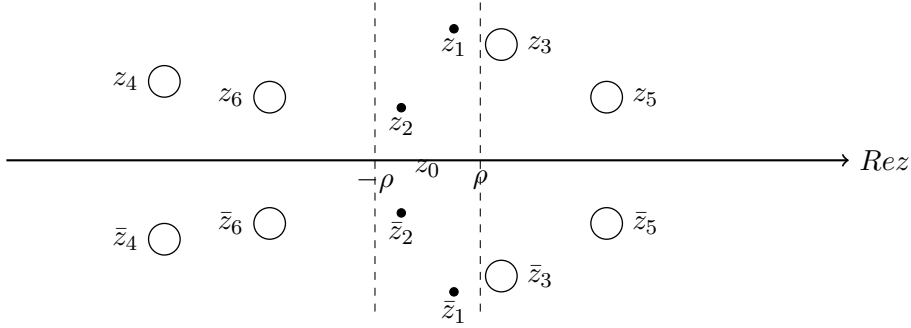
\centerline{\noindent {\small \textbf{Figure 2.} The contour of $\Sigma^{(1)}$.}}

Then, $M^{(1)}(z)$ satisfies the following matrix RHP.
\begin{RHP}\label{RH-3}
Find a matrix function $M^{(1)}$ with the following properties:
\begin{itemize}
  \item $M^{(1)}(y,t,z)$ is meromorphic in $C\setminus \Sigma^{(1)}$;
  \item $M^{(1)}(y,t,z)= I+O(z^{-1})$ as $z\rightarrow \infty$;
  \item For $z\in \Sigma^{(1)}$, the boundary values $M^{(1)}_{\pm}(z)$ satisfy the jump relationship $M^{(1)}_{+}(z)=M^{(1)}_{-}(z)V^{(1)}(z)$, where
      \begin{align}\label{3-7}
       V^{(1)}=\left\{\begin{aligned}
      \left(
        \begin{array}{cc}
      1 & 0 \\
      \bar{r}(z)T(z)^{2}e^{it\theta(z)} & 1 \\
        \end{array}
      \right)\left(
     \begin{array}{cc}
       1 & r(z)T(z)^{-2}e^{-it\theta(z)} \\
       0 & 1 \\
      \end{array}
    \right), ~~z\in(z_{0}, \infty),\\
   \left(
    \begin{array}{cc}
    1 &  \frac{r(z)T_{-}(z)^{-2}}{1+|r(z)|^{2}}e^{-it\theta(z)} \\
    0 & 1 \\
     \end{array}
   \right)\left(
    \begin{array}{cc}
    1 & 0 \\
    \frac{\bar{r}(z)T_{+}(z)^{2}}{1+|r(z)|^{2}}e^{it\theta(z)} & 1 \\
   \end{array}
  \right),~~z\in(-\infty, z_{0}),\\
   \left(
        \begin{array}{cc}
      1 & -\frac{c_{j}e^{-i\theta(z_{j})}}{z-z_{j}}T(z)^{-2} \\
      0 & 1 \\
        \end{array}
      \right), ~~|z-z_{j}|=\rho,~~z_{j}\in\mathcal{Z}^{-}(I),\\
   \left(
        \begin{array}{cc}
      1 & 0 \\
      -\frac{z-z_{j}}{c_{j}e^{-i\theta(z_{j})}}T(z)^{2} & 1 \\
        \end{array}
      \right), ~~|z-z_{j}|=\rho,~~z_{j}\in\mathcal{Z}^{+}(I),\\
    \left(
        \begin{array}{cc}
      1 & 0 \\
      -\frac{\bar{c}_{j}e^{i\theta(\bar{z}_{j})}}{z-\bar{z}_{j}}T(z)^{2} & 1 \\
        \end{array}
      \right), ~~|z-z_{j}|=\rho,~~\bar{z}_{j}\in\mathcal{Z}^{-}(I),\\
      \left(
        \begin{array}{cc}
      1 &  -\frac{z-\bar{z}_{j}}{\bar{c}_{j}e^{i\theta(\bar{z}_{j})}}T(z)^{-2} \\
     0 & 1 \\
        \end{array}
      \right), ~~|z-\bar{z}_{j}|=\rho,~~\bar{z}_{j}\in\mathcal{Z}^{+}(I).
   \end{aligned}\right.
   \end{align}
   \item If $(x,t)$ are such that there exist $n\in\{1,2,\ldots,N\}$ such that $Rez_{n}-z_{0}|<\rho$, $z{0}=-\frac{y}{2t}\Phi_{0}$, then $M^{(1)}(z)$ has simple poles at each $z_{n}\in \mathcal{Z}$ and $\bar{z}_{k}\in \bar{\mathcal{Z}}$ at which
\begin{align}\label{ResM1}
\begin{split}
\mathop{Res}\limits_{z=z_{n}}M^{(1)}=\left\{\begin{aligned}
&\lim_{z\rightarrow z_{n}}M^{(1)}\left(\begin{array}{cc}
    0 & \\
    c_{n}^{-1}\left((\frac{1}{T})'(z_{n})\right)^{-2}e^{it\theta(z)} & 0 \\
  \end{array}
\right),n\in \Delta_{z_{0}}^{-},\\
&\lim_{z\rightarrow z_{k}}M^{(1)}\left(
  \begin{array}{cc}
    0 & c_{n}T^{-2}(z_{n})e^{-it\theta(z)} \\
     0 & 0 \\
  \end{array}\right),k\in \Delta_{z_{0}}^{+},
\end{aligned}\right.\\
\mathop{Res}\limits_{z=\bar{z}_{k}}M^{(1)}=\left\{\begin{aligned}
&\lim_{z\rightarrow \bar{z}_{n}}M^{(1)}\left(\begin{array}{cc}
    0 & -\bar{c}_{n}^{-1}(T'(\bar{z}_{n}))^{-2}e^{-it\theta(z)}\\
    0 & 0 \\
  \end{array}
\right),k\in \Delta_{z_{0}}^{-},\\
&\lim_{z\rightarrow \bar{z}_{n}}M^{(1)}\left(
  \begin{array}{cc}
    0 & -\bar{c}_{n}(T(\bar{z}_{n}))^{2}e^{it\theta(z)} \\
    0 & 0 \\
  \end{array}\right),k\in \Delta_{z_{0}}^{+}.
\end{aligned}\right.
\end{split}
\end{align}
\end{itemize}
\end{RHP}

\begin{proof}
By referring to the definition of $M^{(1)}(y,t,z)$, Proposition \ref{T-property} and the properties of $M(y,t;z)$, the analyticity, jump matrix and asymptotic behavior of $M^{(1)}$ can be obtained directly. Moreover, a similar calculation can show the results of the residue condition of $M^{(1)}(y,t,z)$ (For detail, see \cite{AIHP}).
\end{proof}

Then, the solution of the WKI equation \eqref{WKI-equation} can be expressed as
\begin{align}\label{q-M1}
q(x,t)=q(y(x,t),t)=q_{\pm}+ \frac{(\Phi_{0}+1)^{4}}{(\Phi_{0}+1)^{4}-q_{0}^{4}}\mathcal{M}^{(1)}(x,t),~~y(x,t)=x-c_{-},\\
c_{-}=\lim_{z\rightarrow0}\frac{1}{iz}((M^{(1)}(y,t,0)^{-1}M^{(1)}(y,t,z))_{11}-1) -\frac{1}{\Phi_{0}+1}(\bar{q}_{\pm}\mathcal{M}^{(1)}(x,t)+q_{\pm}\bar{\mathcal{M}}^{(1)}(x,t))
\end{align}
where
\begin{align*}
\mathcal{M}^{(1)}(x,t)=\left(\lim_{z\rightarrow0}\frac{\partial}{\partial y}\frac{1}{z}(M^{(1)}(y,t,0)^{-1}M^{(1)}(y,t,z))_{12}e^{2d} \frac{4\Phi_{0}+^{2}}{(\Phi_{0}+1)}\right.\\
 \left.-\frac{q_{\pm}^{2}}{\Phi_{0}+1}\lim_{z\rightarrow0}\frac{\partial}{\partial y}\frac{1}{z}(M^{(1)}(y,t,0)^{-1}M^{(1)}(y,t,z))_{21}e^{-2d} \frac{4\Phi_{0}+^{2}}{(\Phi_{0}+1)}\right).
\end{align*}

\section{Opening $\bar{\partial}$-lenses and mixed $\bar{\partial}$-RH problem}

In this section, our purpose is to remove the jump from the real axis in new lines along which the $\exp(it\theta(z))$ is decay/growth for $z\notin \mathbb{R}$. Meanwhile, we intend to open the lens in such a way that the lenses are bounded away from the disks to remove the poles from the problem.

To approach this purposes, we first fix a sufficiently small angle $\theta_0$ such that the set $\{z\in\mathbb{C}: ~\cos\theta_0<\big|\frac{Rez}{z}\big|\}$ does not intersect any of the disks $|z-z_{k}|\leqslant\rho$.
Define
\begin{align*}
\phi=\min\{\theta_{0},~\frac{\pi}{4}\},
\end{align*}
and  $\Omega=\bigcup_{k=1}^{4}\Omega_{k}$ where
\begin{align*}
\Omega_{1}=\{z: \arg(z-z_{0})\in(0,\phi)\}, ~~\Omega_{2}=\{z: \arg(z-z_{0})\in(\pi-\phi,\pi)\},\\
\Omega_{3}=\{z: \arg(z-z_{0})\in(-\pi,-\pi+\phi)\}, ~~\Omega_{4}=\{z: \arg(z-z_{0})\in(-\phi,0)\},.
\end{align*}
Denote
\begin{align}\label{Sigma}
\begin{split}
\Sigma_{1}=z_{0}+e^{i\phi}\mathbb{R}_{+},~~\Sigma_{2}=z_{0}+e^{i(\pi-\phi)}\mathbb{R}_{+},\\
\Sigma_{3}=z_{0}+e^{-i(\pi-\phi)}\mathbb{R}_{+},~~\Sigma_{4}=z_{0}+e^{-i\phi}\mathbb{R}_{+}.
\end{split}
\end{align}

\centerline{\begin{tikzpicture}[scale=0.7]
\path [fill=pink] (0,0)--(-8,-1) to (-8,1) -- (0,0);
\path [fill=pink] (0,0)--(8,1) to (8,-1) -- (0,0);
\draw[->][dashed](-8,0)--(8,0)node[right]{$Rez$};
\draw[-][thick](-8,1)--(8,-1);
\draw[-][thick](-8,-1)--(8,1);
\draw[fill] (0,0.2)node[below]{$z_{0}$};
\draw[fill] (1,0)node[below]{$\rho$};
\draw[fill] (-1,0)node[below]{$-\rho$};
\draw[-][dashed](-1,3)--(-1,-3);
\draw[-][dashed](1,3)--(1,-3);
\draw[fill] (0.5,2.5)node[below]{$z_{1}$} circle [radius=0.08];
\draw[fill] (0.5,-2.5)node[below]{$\bar{z}_{1}$} circle [radius=0.08];
\draw[fill] (-0.5,1)node[below]{$z_{2}$} circle [radius=0.08];
\draw[fill] (-0.5,-1)node[below]{$\bar{z}_{2}$} circle [radius=0.08];
\draw[fill] (1.7,2.2)node[right]{$z_{3}$} ;
\draw[fill] (1.7,-2.2)node[right]{$\bar{z}_{3}$};
\draw[fill] (-5.3,1.5)node[left]{$z_{4}$};
\draw[fill] (-5.3,-1.5)node[left]{$\bar{z}_{4}$};
\draw[fill] (3.7,1.2)node[right]{$z_{5}$};
\draw[fill] (3.7,-1.2)node[right]{$\bar{z}_{5}$};
\draw[fill] (-3.3,1.2)node[left]{$z_{6}$};
\draw[fill] (-3.3,-1.2)node[left]{$\bar{z}_{6}$};
\draw(1.4,2.2) [black, line width=0.5] circle(0.3);
\draw(1.4,-2.2) [black, line width=0.5] circle(0.3);
\draw(-5,1.5) [black, line width=0.5] circle(0.3);
\draw(-5,-1.5) [black, line width=0.5] circle(0.3);
\draw(3.4,1.2) [black, line width=0.5] circle(0.3);
\draw(3.4,-1.2) [black, line width=0.5] circle(0.3);
\draw(-3,1.2) [black, line width=0.5] circle(0.3);
\draw(-3,-1.2) [black, line width=0.5] circle(0.3);
\draw[fill] (6,1)node[right]{$\Sigma_{1}$};
\draw[fill] (6,-1)node[right]{$\Sigma_{4}$};
\draw[fill] (-6,1)node[left]{$\Sigma_{2}$};
\draw[fill] (-6,-1)node[left]{$\Sigma_{3}$};
\draw[fill] (5.8,0)node[below]{$\Omega_{4}$};
\draw[fill] (5.8,0)node[above]{$\Omega_{1}$};
\draw[fill] (-5.8,0)node[below]{$\Omega_{3}$};
\draw[fill] (-5.8,0)node[above]{$\Omega_{2}$};
\end{tikzpicture}}
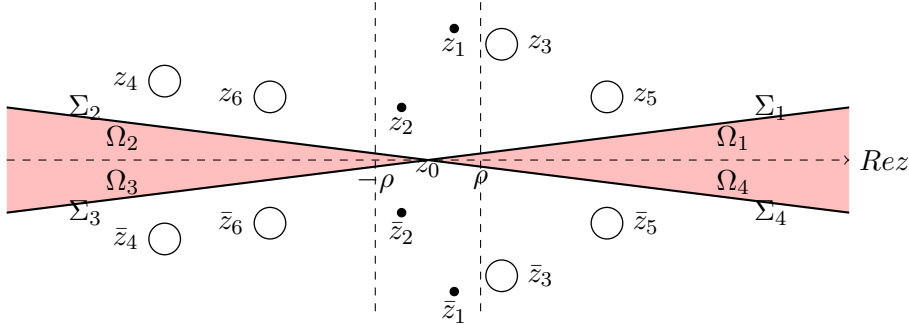
\centerline{\noindent {\small \textbf{Figure 3.} Find a $\phi$ sufficiently small so that there is no pole in the cone and the four rays}}
\centerline{$\Sigma_k, k=1,2,3,4$ can not  intersect with any disks $|z\pm z_k|<\rho$ or $|z\pm \bar{z}_k|<\rho$.}

Moreover, define $\chi_{Z} \in C_{0}^{\infty} (C, [0, 1])$ is supported near the discrete spectrum such that
\begin{align}\label{4-3}
\chi_{Z}(z)=\left\{\begin{aligned}
&1,~~dist(z,\mathcal{Z}\cup \bar{\mathcal{Z}} )<\rho/3, \\
&0,~~dist(z,\mathcal{Z}\cup \bar{\mathcal{Z}} )>2\rho/3.
\end{aligned}\right.
\end{align}
Then, we define extensions of the jump matrices of \eqref{3-7} in the following proposition.

\begin{prop}\label{R-property}
There exist functions $R_{j}: \bar{\Omega}_{j} \rightarrow C, j= 1, 3, 4, 6$ with boundary values such that
\begin{align*}
&R_{1}(z)=\left\{\begin{aligned}&r(z)T^{-2}(z), ~~~~z\in(z_{0}, \infty),\\
&r(z_{0})T_{0}^{-2}(z_{0})(z-z_{0})^{-2i\nu(z_{0})}(1-\chi_{Z}(z)), z\in\Sigma_{1},
\end{aligned}\right.\\
&R_{2}(z)=\left\{\begin{aligned}&\frac{\bar{r}(z)}{1+|r(z)|^{2}}T_{+}^{2}(z), ~~~~z\in(-\infty, z_{0}),\\
&\frac{\bar{r}(z_{0})}{1+|r(z_{0})|^{2}}T_{0}^{2}(z_{0}) (z-z_{0})^{2i\nu(z_{0})}(1-\chi_{Z}(z)), z\in\Sigma_{2},
\end{aligned}\right.\\
&R_{3}(z)=\left\{\begin{aligned}&\frac{r(z)}{1+|r(z)|^{2}}T_{-}^{-2}(z), ~~~~z\in(-\infty, z_{0}),\\
&\frac{r(z_{0})}{1+|r(z_{0})|^{2}}T_{0}^{-2}(z_{0}) (z-z_{0})^{-2i\nu(z_{0})}(1-\chi_{Z}(z)), z\in\Sigma_{3},
\end{aligned}\right.\\
&R_{4}(z)=\left\{\begin{aligned}&\bar{r}(z)T^{2}(z), ~~~~z\in(z_{0}, \infty),\\
&\bar{r}(z_{0})T_{0}^{2}(z_{0})(z-z_{0})^{2i\nu(z_{0})}(1-\chi_{Z}(z)), z\in\Sigma_{4}.
\end{aligned}\right.
\end{align*}
For a fixed constant $c=c(q_{0})$ and cutoff function $\chi_{Z}(z)$ defined in \eqref{4-3}, $R_{j}$ satisfy that
\begin{align}\label{R-estimate}
\begin{split}
&|R_{j}(z)|\leq c\sin^{2}(\arg \frac{\pi}{2\theta_{0}}(z-z_{0}))+\left<Rez\right>^{-1/2},\\
&|\bar{\partial}R_{j}(z)|\leq c\bar{\partial}\chi_{Z}(z)+c|z-z_{0}|^{-1/2}+c|r'(Rez)|,\\
&\bar{\partial}R_{j}(z)=0,z\in elsewhere,~~or~ dist(z,\mathcal{Z}\cup\bar{\mathcal{Z}})\leq \rho/3,
\end{split}
\end{align}
where $\left<Rez\right>=\sqrt{1+(Rez)^{2}}$.\\
\end{prop}
\begin{proof}
We give details for $R_{1}(z)$. The other cases can be proved similarly.
Firstly, according to the construction of $R_{1}(z)$, the extension of  $R_{1}(z)$ can be written as
\begin{align}\label{R-prove-1}
R_{1}(z)=\left(r(Rez)\cos(k\arg(z-z_{0}))\right)+ \left(1-\cos(k\arg(z-z_{0}))g_{1}\right)T^{-2}(z)(1-\chi_{Z}(z)),
\end{align}
where \begin{align*}
        k:=\frac{\pi}{2\theta_{0}},~~g_1:=r(z_0)T_0^{-2}(z)T^{2}(z)(z-z_{0})^{-2i\nu(z_0)}.
      \end{align*}
Based on the definition of $T_0(z)$ and $T(z)$, we know that $T_0(z)$ and $T(z)$ are bounded for any $z\in\omega_{1}$. Meanwhile, note a fact that $r\in  H^{1}(\mathbb{R})$ implies that $r$ is H\"{o}lder $-\frac{1}{2}$ continuous and $r(Rez)\lesssim|1+(Rez)^{2}|^{-\frac{1}{4}}$. Therefore, $R_{1}(z)$ has the following estimate.
\begin{align*}
|R_{1}(z)|&=|\left(g_{1}+(r(Rez)-g_{1})\cos(k\arg(z-z_{0}))\right)T^{-2}(z)(1-\chi_{Z}(z))|\\
&\lesssim|r(Rez)||T^{-2}(z)(1-\chi_{Z}(z))\cos(k\arg(z-z_{0}))|+ 2|g_{1}||\sin^{2}(\frac{k}{2}\arg(z-z_{0}))T^{-2}(z)(1-\chi_{Z}(z))|\\
&\lesssim|r(Rez)|+|\sin^{2}(\frac{k}{2}\arg(z-z_{0}))|.
\end{align*}
Let $z-z_{0}=se^{i\phi}$. Since $\bar{\partial}=\frac{1}{2}(\partial_{x}+i\partial_{y})= \frac{1}{2}e^{i\phi}(\partial_{s}+is^{-1}\partial_{\phi})$, we have
\begin{align*}
\bar{\partial}&R_{1}(z)=-\left(g_{1}+(r(Rez)-g_{1})\cos(k\arg(z-z_{0}))\right)T^{-2}(z)\bar{\partial}\chi_{Z}(z)\\
&+\left[\frac{1}{2}r'(Rez)\cos(k\arg(z-z_{0}))e^{i\phi} -\frac{(r(Rez)-g_{1})\sin(k\arg(z-z_{0}))}{|z-z_0|}e^{i\phi}\right]T^{-2}(z)(1-\chi_{Z}(z)).
\end{align*}
By using the continuity and decay of $r(Rez)$ analysed above and the fact that when $1-\chi_{Z}(z)$ and $\bar{\partial}\chi_{Z}(z)$ are supported away from the discrete spectrum, the poles and zeros of $T(z)$ have no effect on the bound, which confirm that the first two terms are in bound. For the last term, we observe that
\begin{align*}
 |r(Rez)-g_{1}|\leqslant|r(Rez)-r(z_0)|+|r(z_0)-g_{1}|.
\end{align*}
Furthermore, we use Cauchy-Schwarz and the Proposition \ref{T-property}, the boundedness of the last term can be confirmed easily. Then, the bound \eqref{R-estimate} for $z\in\Omega_1$ follows immediately.
\end{proof}

Now, we use the extension of Proposition \ref{R-property} to define modified versions of the factorization of $V^{(1)}$. As a result, on the real axis, we have
\begin{align*}
V^{(1)}(z)=\hat{b}^{-\dag}(z)\hat{b}(z)=\hat{B}(z)\hat{B}^{-\dag}(z),
\end{align*}
where
\begin{align*}
\hat{b}^{\dag}(z)=\left(
                    \begin{array}{cc}
                      1 & -R_3(z)e^{-it\theta(z)} \\
                      0 & 1 \\
                    \end{array}
                  \right),~~ \hat{b}(z)=\left(
                    \begin{array}{cc}
                      1 & 0 \\
                      R_2(z)e^{it\theta(z)} & 1 \\
                    \end{array}
                  \right),\\
\hat{B}(z)=\left(
                    \begin{array}{cc}
                      1 & 0 \\
                      R_4(z)e^{it\theta(z)} & 1 \\
                    \end{array}
                  \right),~~ \hat{B}^{\dag}(z)=\left(
                    \begin{array}{cc}
                      1 & -R_1(z)e^{-it\theta(z)} \\
                      0 & 1 \\
                    \end{array}
                  \right).
\end{align*}
Then, we can define the new transformation
\begin{align}\label{Trans-2}
M^{(2)}(z)=M^{(1)}(z)R^{(2)}(z),
\end{align}
where
\begin{align}
R^{(2)}(z)=\left\{\begin{aligned}
&\hat{B}^{\dag}(z), ~&z\in\Omega_{1},\\
&\hat{b}^{-1}(z), ~&z\in\Omega_{2},\\
&\hat{b}^{-\dag}(z), ~&z\in\Omega_{3},\\
&\hat{B}(z),~ &z\in\Omega_{4},\\
&\mathbb{I},~ &z\in\mathbb{C}\setminus\bar{\Omega}.
\end{aligned}
\right.
\end{align}
Let
\begin{align*}
\Sigma^{(2)}=\bigcup_{z_{j}\in\mathcal{Z}(I)}\left\{z\in\mathbb{C}: |z-z_j|=\rho~ or~ |z-z^{*}_j|=\rho \right\}\cup\Sigma_{1}\cup\Sigma_{2}\cup\Sigma_{3}\cup\Sigma_{4}.
\end{align*}
Then, $M^{(2)}(z)$ satisfies the following mixed $\bar{\partial}$-RH problem.

\begin{RHP}\label{RH-4}
Find a matrix value function $M^{(2)}$ satisfying
\begin{itemize}
 \item $M^{(2)}(y,t,z)$ is continuous in $\mathbb{C}\backslash(\Sigma^{(2)})$.
 \item $M^{(2)}(x,t,z)=\mathbb{I}+O(z^{-1}), $ as \quad $z\rightarrow\infty$.
 \item $M_+^{(2)}(y,t,z)=M_-^{(2)}(y,t,z)V^{(2)}(y,t,z),$ \quad $z\in\Sigma^{(2)}$, where the jump matrix $V^{(2)}(x,t,z)$ satisfies
 \begin{align}\label{J-V2}
V^{(2)}(z)=\left\{\begin{aligned}
&\begin{pmatrix}1&r(z_0)T_0(z_0)^{-2}(z-z_0)^{-2i\nu(z_0)}e^{-it\theta(z_0)}(1-\chi_{Z}(z))\\0&1\end{pmatrix} &z\in\Sigma_1, \\
&\begin{pmatrix}1&0\\ \frac{\bar{r}(z_0)T_0(z_0)^2}{1+|{r(z_0})|^2}(z-z_0)^{i\nu(z_0)}e^{2it\theta(z_0)}(1-\chi_{Z}(z))&1\end{pmatrix} &z\in\Sigma_2, \\
&\begin{pmatrix}1&\frac{r(z_0)T_0(z_0)^{-2}}{1+|{r(z_0})|^2}(z-z_0)^{-i\nu(z_0)} e^{-2it\theta(z_0)}(1-\chi_{Z}(z))\\0&1\end{pmatrix} &z\in\Sigma_3, \\
&\begin{pmatrix}1&0\\\bar{r}(z_0)T_0(z_0)^{2}(z-z_0)^{2i\nu(z_0)}e^{it\theta(z_0)} (1-\chi_{Z}(z))&1\end{pmatrix} &z\in\Sigma_4,\\
&\begin{pmatrix}1&-(z-z_j)^{-1}c_jT(z)^{-2}e^{-it\theta(z_j)}\\0&1\end{pmatrix} &|z-z_j|=\rho,~z_j\in\mathcal{Z}^{-}(I), \\
&\begin{pmatrix}1&0\\ -(z-z_j)c^{-1}_jT(z)^{2}e^{it\theta(z_j)}&1\end{pmatrix} &|z-z_j|=\rho,~z_j\in\mathcal{Z}^{+}(I), \\
&\begin{pmatrix}1&0\\ (z-z^{*}_j)^{-1}c^{*}_jT(z)^{-2}e^{it\theta(z^{*}_j)}&1\end{pmatrix} &|z-z_j|=\rho,~z^{*}_j\in\mathcal{Z}^{-}(I), \\
&\begin{pmatrix}1&(z-z^{*}_j)c^{*,-1}_jT(z)^{-2}e^{-it\theta(z^{*}_j)}\\0&1\end{pmatrix} &|z-z_j|=\rho,~z^{*}_j\in\mathcal{Z}^{+}(I).
\end{aligned}\right.
\end{align}
\item For $\mathbb{C}$, $\bar{\partial}M^{(2)}=M^{(2)}W(z),$ where
    \begin{align}\label{4.1}
W(z)=\bar{\partial} R^{(2)}(z)=\left\{\begin{aligned}
&\begin{pmatrix}0&-\bar{\partial}R_1 e^{-it\theta}\\0&0\end{pmatrix}, &z\in\Omega_1, \\
&\begin{pmatrix}0&0\\-\bar{\partial}R_2 e^{it\theta}&0\end{pmatrix}, &z\in\Omega_2, \\
&\begin{pmatrix}0&\bar{\partial}R_3 e^{-it\theta}\\ 0&0\end{pmatrix}, &z\in\Omega_3, \\
&\begin{pmatrix}0&0\\ \bar{\partial}R_4 e^{it\theta}&0\end{pmatrix}, &z\in\Omega_4,\\
&\begin{pmatrix}0&0\\0&0\end{pmatrix}, &z\in elsewhere.
\end{aligned}\right.
\end{align}
  \item  If $z_j\in \mathcal{Z}^{\pm}(I)$ for $j=1,2,\ldots,N$, then $M^{(2)}$ is analytic in the region $\mathbb{C}\backslash(\bar{\Omega}\cup\Sigma^{(2)}$. If there exist $z_j\in \mathcal{Z}(I)$, i.e., $Rez_j\in(-\rho,\rho)$, then $M^{(2)}(z)$ admits the residue conditions at poles $z_{k} \in \mathcal{Z}$ and $\bar{z}_{k} \in \bar{\mathcal{Z}}$, i.e.,
      \begin{align}
j\in\triangle_{z_0}^{-}:\left\{\begin{aligned}
&\mathop{Res}_{z=z_j}M^{(2)}(z)=\mathop{lim}_{z\rightarrow z_j}M^{(2)}(z)\left(
                                                                           \begin{array}{cc}
                                                                             0 & 0 \\
c_j^{-1}((\frac{1}{T})^{'}(z_j))^{-2}e^{it\theta(z_j)} & 0 \\
                                                                           \end{array}
                                                                         \right)
,\\
&\mathop{Res}_{z=z^{*}_j}M^{(2)}(z)=\mathop{lim}_{z\rightarrow z_j}M^{(2)}(z)\left(
                                                                           \begin{array}{cc}
0 & c_j^{*}(T^{'}(z^{*}_j))^{-2}e^{-it\theta(z^{*}_j)} \\
0 & 0 \\
                                                                           \end{array}
                                                                         \right),
\end{aligned}
\right.\\
j\in\triangle_{z_0}^{+}:\left\{\begin{aligned}
&\mathop{Res}_{z=z_j}M^{(2)}(z)=\mathop{lim}_{z\rightarrow z_j}M^{(2)}(z)\left(
                                                                           \begin{array}{cc}
0 & c_j(T(z_j))^{-2}e^{-it\theta(z_j)} \\
0 & 0 \\
                                                                           \end{array}
                                                                         \right)
,\\
&\mathop{Res}_{z=z^{*}_j}M^{(2)}(z)=\mathop{lim}_{z\rightarrow z_j}M^{(2)}(z)\left(
                                                                           \begin{array}{cc}
0 & 0 \\
-c_j^{*}(T(z^{*}_j))^{2}e^{it\theta(z^{*}_j)} & 0 \\
                                                                           \end{array}
                                                                         \right),
\end{aligned}
\right.\\
\end{align}
\end{itemize}
\end{RHP}

Then, the solution of the WKI equation \eqref{WKI-equation} can be expressed as
\begin{align}\label{q-M2}
&q(x,t)=q(y(x,t),t)=q_{\pm}+ \frac{(\Phi_{0}+1)^{4}}{(\Phi_{0}+1)^{4}-q_{0}^{4}}\mathcal{M}^{(2)}(x,t), y(x,t)=x-c_{-},\\
&c_{-}=\lim_{z\rightarrow0}\frac{1}{iz}((M^{(2)}(y,t,0)^{-1}M^{(2)}(y,t,z))_{11}-1) -\frac{1}{\Phi_{0}+1}(\bar{q}_{\pm}\mathcal{M}^{(2)}(x,t)+q_{\pm}\bar{\mathcal{M}}^{(2)}(x,t))
\end{align}
where
\begin{align*}
\mathcal{M}^{(2)}(x,t)=\left(\lim_{z\rightarrow0}\frac{\partial}{\partial y}\frac{1}{z}(M^{(2)}(y,t,0)^{-1}M^{(2)}(y,t,z))_{12}e^{2d} \frac{4\Phi_{0}^{2}}{(\Phi_{0}+1)}\right.\\
 \left.-\frac{q_{\pm}^{2}}{\Phi_{0}+1}\lim_{z\rightarrow0}\frac{\partial}{\partial y}\frac{1}{z}(M^{(2)}(y,t,0)^{-1}M^{(2)}(y,t,z))_{21}e^{-2d} \frac{4\Phi_{0}^{2}}{(\Phi_{0}+1)}\right).
\end{align*}

\section{Decomposition of the mixed $\bar{\partial}$-RH problem}

In order to obtain the solution $M^{(2)}$ of RHP \ref{RH-4}, we  decompose the mixed $\bar{\partial}$-RH problem \ref{RH-4} into two parts:  a model RH problem for $M^{(2)}_{R}(y,t,z)\triangleq M^{(2)}_{R}(z)$ with  $\bar{\partial}R^{(2)}=0$ and a pure $\bar{\partial}$-RH problem with $\bar{\partial}R^{(2)}\neq0$. For the first step, we construct a solution $M^{(2)}_{R}$ to the model RH problem as follows.

\begin{RHP}\label{RH-rhp}
Find a matrix value function $M^{(2)}_{R}$, admitting
\begin{itemize}
 \item $M^{(2)}_{R}$ is continuous in $\mathbb{C}\backslash(\Sigma^{(2)}$;
 \item $M^{(2)}_{R,+}(y,t,z)=M^{(2)}_{R,-}(y,t,z)V^{(2)}(y,t,z),$ \quad $z\in\Sigma^{(2)}$, where $V^{(2)}(y,t,z)$ is the same with the jump matrix appearing in RHP \ref{RH-4};
 \item $M^{(2)}_{R}(y,t,z)=I+o(z^{-1})$ as $z\rightarrow\infty$, ;
 \item $M^{(2)}_{R}$ possesses the same residue condition with $M^{(2)}$.
 \end{itemize}
\end{RHP}

In order to facilitate the subsequent analysis, we first give the estimation of the jump matrix $V^{(2)}(y,t,z)$.

\begin{prop}\label{V2-estimate-prop}
As $|t|\rightarrow\infty$, for $1\leq p<\infty$, the jump matrix $V^{(2)}(y,t,z)$ satisfies
\begin{align}\label{V2-estimate-1}
\big|\big|V^{(2)}-\mathbb{I}\big|\big|_{L^{p}(\bigcup_{j=1}^{4}\Sigma_j)}=\left\{\begin{aligned}
&O(e^{-|t|\frac{\varepsilon}{4\Phi_0}\rho^{2}}), &z\in\bigcup_{j=1}^{4}\Sigma_j\setminus U_{z_{0}},\\
&O(|z-z_0|^{-1}t^{\frac{1}{2}}), &z\in\bigcup_{j=1}^{4}\Sigma_j\cap U_{z_{0}},
\end{aligned}\right.
\end{align}
\begin{align}\label{V2-estimate-2}
\big|\big|V^{(2)}-\mathbb{I}\big|\big|_{L^{p}(\Sigma^{(2)}\setminus\bigcup_{j=1}^{4}\Sigma_j)} =O(e^{-|t|\frac{2}{\Phi_0}\rho^{2}}),
\end{align}
where $\varepsilon$ is a sufficiently small  positive constant and $U_{z_{0}}=\{z: |z-z_0|<\frac{\rho}{2}\}.$
\end{prop}
\begin{proof}
For the estimation \eqref{V2-estimate-1}, we prove the case $z\in\Sigma_1$. The other case can be proved similarly. For any $z\in\Sigma_1$, we have $z-z_0=|z-z_0|e^{i\arg(z-z_0)}$.
So, $\big|V^{(2)}-\mathbb{I}\big|$ can be written as
\begin{align*}
\big|V^{(2)}-\mathbb{I}\big|=\big|R_1|e^{-t\frac{1}{\Phi_{0}}|z-z_0|^{2}\sin(2\arg(z-z_0))}.
\end{align*}
Observe that $\arg(z-z_0)\in(0,\frac{\pi}{4})$, there exists a sufficiently small  positive constant $\varepsilon$ such that $\sin(2\arg(z-z_0))\in (\varepsilon,1)$. Moreover, by using Proposition \ref{R-property}, for $z\in\bigcup_{j=1}^{4}\Sigma_j\setminus U_{z_{0}}$, the first estimation in \eqref{V2-estimate-1} can be obtained immediately.

Following the similar idea, the remaining two estimation can be obtained easily.
\end{proof}

Observe that the first estimation in \eqref{V2-estimate-1} implies that the jump matrix $V^{(2)}$ uniformly approaches to $\mathbb{I}$ on $\bigcup_{j=1}^{4}\Sigma_j\setminus U_{z_{0}}$. As a result, if we omit the jump condition of $M^{(2)}_{R}$ outside the $U_{z_{0}}$, there  is only exponentially decay error (in $t$). So the solution $M^{(2)}_{R}$ can be constructed as
\begin{align}\label{Mrhp}
M^{(2)}_{R}(z)=\left\{\begin{aligned}
&E(z)M^{sol}(z), &&z\in\mathbb{C}\backslash U_{z_0},\\
&E(z)M^{sol}(z)M^{(pc)}(z_0,r_0), &&z\in U_{z_0}.
\end{aligned} \right.
\end{align}
In this decomposition , $M^{sol}$ solves a model RHP by omitting the jump conditions, $M^{(pc)}$ is a known parabolic cylinder model and $E(z)$, an error function, is the solution of a small-norm Riemann-Hilbert problem.

In the following part, the proof of the existence and asymptotic of $M^{(2)}_{R}$ will given .  So we can use $M^{(2)}_{R}$ to construct  a new matrix function
\begin{align}\label{delate-pure-RHP}
M^{(3)}(z)=M^{(2)}(z)M^{(2)}_{R}(z)^{-1}.
\end{align}
Then $M^{(3)}(z)$ satisfies the following pure $\bar{\partial}$-RH problem.

\begin{RHP}\label{RH-5}
Find a matrix value function $M^{(3)}$ admitting
\begin{itemize}
 \item $M^{(3)}$ is continuous with sectionally continuous first partial derivatives in $\mathbb{C}$;
 \item For $z\in \mathbb{C}$, we obtain $\bar{\partial}M^{(3)}(z)=M^{(3)}(z)W^{(3)}(z)$,
       where
       \begin{align}\label{5-1}
       W^{(3)}=M^{(2)}_{R}(z)\bar{\partial}R^{(2)}M^{(2)}_{R}(z)^{-1};
       \end{align}
 \item As $z\rightarrow\infty$,
       \begin{align}
       M^{(3)}(z)=\mathbb{I}+o(z^{-1}).
       \end{align}
 \end{itemize}
\end{RHP}
\begin{proof}
Based on the definition of $M^{(2)}_{R}$ and the above Propositions, following the idea in \cite{AIHP}, the results of this RHP\ref{RH-5} can be derived easily.
\end{proof}

Moreover, the existence , uniqueness and asymptotic of $M^{(3)}$ will be shown in later Section.

\section{Asymptotic $N$-soliton solution}

In this section, we are going to remove the Riemann-Hilbert component of the solution, so that all that remains is $M^{(3)}(z)$.

Let $M^{sol}$ denote the solution of RHP\ref{RH-4} corresponding to the case $\bar{\partial}R^{(2)}=0$. As a result, $M^{sol}$ satisfies the following RHP.

\begin{RHP}\label{RH-6}
Find a matrix value function $M^{sol}$ satisfying
\begin{itemize}
  \item $M^{sol}(y,t;z)$ is continuous  in $\mathbb{C}\backslash\left((\Sigma^{(2)}\setminus\bigcup_{j=1}^{4}\Sigma_{j})\cup U_{z_0}\right)$;
  \item $M^{sol}_{+}(y,t,z)=M^{sol}_{-}(y,t,z)V^{(2)}(y,t,z),$ \quad $z\in\Sigma^{(2)}\backslash U_{z_0}$, where $V^{(2)}(y,t,z)$ is the same with the jump matrix appearing in RHP \ref{RH-4};
  \item As $z\rightarrow\infty$,
       \begin{align}
       M^{sol}(y,t;z)=\mathbb{I}+o(z^{-1});
       \end{align}
  \item $M^{sol}(y,t;z)$  admits the same residue condition in RHP \ref{RH-4} with $M^{sol}(y,t;z)$ replacing $M^{(2)}(y,t;z)$.
\end{itemize}
\end{RHP}

Next, we briefly show a fact that for any admissible scattering data $\{r(z), \{z_j, c_j\}_{j=1}^{N}\}$ in RHP\ref{RH-4}, the solution  $M^{sol}$ exists and  is equivalent to a reflectionless solution of the RHP\ref{RH-2}, by a transformation, with the modified scattering data $\{0, \{z_j, \tilde{c}_j\}_{j=1}^{N}\}$ where the coefficients $\tilde{c}_j$ are that
\begin{align}\label{c-c}
\tilde{c}_j=c_je^{\frac{1}{\pi i}\int_{z_{0}}^{\infty}\frac{\log(1+|r(s)|^{2})}{s-z_j}\,ds}.
\end{align}

By using \eqref{T-define} and \eqref{Trans-1}, we introduction the following transformation
\begin{align}\label{M-tilde}
\tilde{M}(z)=M^{sol}H(z) \left(\prod_{k\in\Delta_{z_{0}}^{+}}\frac{z-\bar{z}_{k}}{z-z_{k}}\right)^{-\sigma_3},
\end{align}
where
\begin{align*}
H(z)=\left\{\begin{aligned}
      \left(
        \begin{array}{cc}
      1 & -\frac{c_{j}e^{-i\theta(z_{j})}}{z-z_{j}}T(z)^{-2} \\
      0 & 1 \\
        \end{array}
      \right), ~~|z-z_{j}|<\rho,~~z_{j}\in\mathcal{Z}^{-}(I),\\
   \left(
        \begin{array}{cc}
      1 & 0 \\
      -\frac{z-z_{j}}{c_{j}e^{-i\theta(z_{j})}}T(z)^{2} & 1 \\
        \end{array}
      \right), ~~|z-z_{j}|<\rho,~~z_{j}\in\mathcal{Z}^{+}(I),\\
    \left(
        \begin{array}{cc}
      1 & 0 \\
      -\frac{\bar{c}_{j}e^{i\theta(\bar{z}_{j})}}{z-\bar{z}_{j}}T(z)^{2} & 1 \\
        \end{array}
      \right), ~~|z-z_{j}|<\rho,~~\bar{z}_{j}\in\mathcal{Z}^{-}(I),\\
     \left(
        \begin{array}{cc}
      1 &  -\frac{z-\bar{z}_{j}}{\bar{c}_{j}e^{i\theta(\bar{z}_{j})}}T(z)^{-2} \\
     0 & 1 \\
        \end{array}
      \right), ~~|z-\bar{z}_{j}|<\rho,~~\bar{z}_{j}\in\mathcal{Z}^{+}(I).
   \end{aligned}\right.
\end{align*}
Then, for $n\in\triangle_{z_0}^{+}$, the residue condition of $\tilde{M}(z)$ at $z_n$ can be derived as
\begin{align*}
\mathop{Res}_{z=z_n}\tilde{M}(z)=\mathop{lim}_{z\rightarrow z_n}\tilde{M}(z)\left(
                                                                           \begin{array}{cc}
0 & c_n(\delta(z_n))^{-2}e^{-it\theta(z_n)} \\
0 & 0 \\
                                                                           \end{array}
                                                                         \right).
\end{align*}
Thus, $\tilde{M}(z)$ is the solution of RHP\ref{RH-2} with scattering data $\{r(z)\equiv0, \{z_j, \tilde{c}_j\}_{j=1}{N}\}$. As a result, $M^{sol}$ is
the solution of RHP\ref{RH-2} corresponding to a $N$-soliton, reflectionless, potential $\tilde{q}(x,t)$ which generates the same discrete spectrum $\mathcal{Z}$ as our initial data, but whose connection coefficients \eqref{c-c} are perturbations of those for the original initial data by an amount related to the reflection coefficient of the initial data.

Now, the solution $M^{sol}$  is  uniqueness and existence. However, not all discrete spectra contribute to the solution as $t\rightarrow\infty$. From the Proposition \ref{V2-estimate-prop}, we know that the jump matrix is uniformly approaching to identity and do not have significant contribution to the asymptotic behavior of the solution. Inspired by this fact, we next consider to completely omit the jump condition on $M^{sol}$. So we decompose the $M^{sol}$ as
\begin{align}\label{Etilde-define}
M^{sol}(z)=\tilde{E}(z)M^{sol}_{I}(z),
\end{align}
where $\tilde{E}(z)$ is a error function, which is a solution of a small-norm RH problem, and $M^{sol}_{I}(z)$ solves RHP\ref{RH-rhp} with $V^{(2)}(z)\equiv0$. Then, RHP\ref{RH-rhp} reduces to the following RHP.

 \begin{RHP}\label{RH-7}
Find a matrix value function $M^{sol}_{I}$ satisfying
\begin{itemize}
  \item $M^{sol}_{I}(y,t;z)$ is continuous  in $\mathbb{C}\backslash\mathcal{Z}$;
  \item As $z\rightarrow\infty$,
       \begin{align}
       M^{sol}_{I}(y,t;z)=I+o(z^{-1});
       \end{align}
  \item Residue condition:
  \begin{align}
\mathop{Res}_{z=z_j}M^{sol}_{I}(z)=\left\{\begin{aligned}
&\mathop{lim}_{z\rightarrow z_j}M^{sol}_{I}(z)\left(
                                                                           \begin{array}{cc}
                                                                             0 & 0 \\
c_j^{-1}((\frac{1}{T})'(z_j))^{-2}e^{it\theta(z_j)} & 0 \\
                                                                           \end{array}
                                                                         \right), ~&z_j\in\triangle_{I}^{-}
,\\
&\mathop{lim}_{z\rightarrow z_j}M^{sol}_{I}(z)\left(
                                                                           \begin{array}{cc}
0 & c_j(T'(z_j))^{-2}e^{-it\theta(z_j)} \\
0 & 0 \\
                                                                           \end{array}
                                                                         \right),~&z_j\in\triangle_{I}^{+}
\end{aligned}
\right.\\
\mathop{Res}_{z=\bar{z}_j}M^{sol}_{I}(z)=\left\{\begin{aligned}
&\mathop{lim}_{z\rightarrow \bar{z}_j}M^{sol}_{I}(z)\left(
                                                                           \begin{array}{cc}
0 & -\bar{c}^{-1}_j(\bar{T}'(z_j))^{-2}e^{-it\theta(\bar{z}_j)} \\
0 & 0 \\
                                                                           \end{array}
                                                                         \right)~&\bar{z}_j\in\triangle_{I}^{-}
,\\
&\mathop{lim}_{z\rightarrow \bar{z}_j}M^{sol}_{I}(z)\left(
                                                                           \begin{array}{cc}
0 & 0 \\
-\bar{c}_j(\bar{T}(\bar{z}_j))^{-2}e^{it\theta(\bar{z}_j)} & 0 \\
                                                                           \end{array}
                                                                         \right),&\bar{z}_j\in\triangle_{I}^{+}.
\end{aligned}
\right.\\
\end{align}
where $\triangle_{I}^{-}=\{z_{k}| Rez_k\in(z_0-\rho, z_0)\}$, and $\triangle_{I}^{+}=\{z_{k}| Rez_k\in(z_0, z_0+\rho)\}$.
\end{itemize}
\end{RHP}

\begin{prop}\label{Msol-prop}
The RHP\ref{RH-7} exists unique solution. Meanwhile, $M^{sol}_{I}(y,t;z)$  is equivalent to a solution of the original RHP\ref{RH-2} with modified scattering data $\sigma_{d}^{sol}=\{r(z)\equiv0, \{z_j, \tilde{c}_j\}_{z_j\in\triangle_{I}^{\pm}}\}$ under the condition that $r(z)\equiv0$. Then, for a fixed $z_0=-\frac{y\Phi_0}{2t}$ and $|t|\gg1$, uniformly for $z\in\mathbb{C}$,  $M^{sol}(y,t;z)$  is expressed as
\begin{align*}
M^{sol}(z)=M^{sol}_{I}(z)\left(\mathbb{I}+O(e^{-|t|\rho^{2}\frac{2}{\Phi_0}})\right),
\end{align*}
and particularly, for $z\rightarrow\infty$, $M^{sol}(y,t;z)$ possesses the asymptotic extension
\begin{align}\label{Msol-extension}
M^{sol}(z)=M^{sol}_{I}(z)\left(\mathbb{I}+z^{-1}O(e^{-|t|\rho^{2}\frac{2}{\Phi_0}})+O(z^{-2})\right).
\end{align}
Moreover, the solution $M^{sol}_{I}(z)$ of the RHP\ref{RH-7} is as follows:
\begin{itemize}
  \item If there does not exist $z_j\in\triangle_{I}^{+}$, then all the $\pm z_j$ are away form the critical lines,
      \begin{align}
        M^{sol}_{I}(z)=\mathbb{I}.
      \end{align}
  \item If there exists $z_{j_{0}}\in\triangle_{I}^{+}$, then
  \begin{align}
   M^{sol}_{z_{j_{0}}}(z)=\mathbb{I}+\frac{1}{z-z_{j_{0}}}\left(
                                       \begin{array}{cc}
                                         0 & \alpha_{j_{0}} \\
                                         0 & \beta_{j_{0}} \\
                                       \end{array}
                                     \right)+\frac{1}{z-\bar{z}_{j_{0}}}\left(
                                       \begin{array}{cc}
                                         \bar{\beta}_{j_{0}} & 0 \\
                                         -\bar{\alpha}_{j_{0}} & 0 \\
                                       \end{array}
                                     \right),\label{Msol-zj0}\\
   \alpha_{j_{0}}=\frac{4Imz_{j_{0}}e^{-\varphi_{j_{0}}}}{1+e^{-\bar{\varphi}_{j_{0}}}},~~
   \bar{\beta}_{j_{0}}= -2ie^{-\bar{\varphi}_{j_{0}}}\frac{4Imz_{j_{0}} e^{-\varphi_{j_{0}}}}{1+e^{-\bar{\varphi}_{j_{0}}}},\notag
  \end{align}
  where the phase $\varphi_{j_{0}}$ is given by
  \begin{align*}
   &\varphi_{j_{0}}=i(z_{j_{0}}(y-y_{j_{0}})+\frac{z^{2}_{j_{0}}}{\Phi_0}t),\\
   &y_{j_{0}}=\frac{1}{iz-{j_{0}}}\left(\log\frac{|c_{j_{0}}|}{4Imz_{j_{0}}} \prod_{k\in\triangle^{+}_{z_{0}},k\notin\triangle^{+}_{I}}\Big| \frac{z_{j_{0}}-z_k}{z_{j_{0}}-\bar{z}_k}\Big|^{2}+\frac{1}{\pi i}\int_{-\infty}^{z_{j_{0}}}\frac{\log(1+|r(s)|^{2})}{s-z_{j_{0}}}\,ds\right).
  \end{align*}
\end{itemize}
\end{prop}
\begin{proof}
Firstly, according to the Liouville's theorem, the uniqueness of solution follows immediately. The case that there does not exist $z_j\in\triangle_{I}^{+}$ can be obtained easily. For the case that there exists $z_{j_{0}}\in\triangle_{I}^{+}$, we just need to substitute \eqref{Msol-zj0} into the residue condition in RHP\ref{RH-7} and then obtain the results.
\end{proof}

Then, denote $q^{sol}(x,t,\sigma_{d}^{sol})$ as the $N(I)$-soliton soluton with scattering data $\sigma_{d}^{sol}=\{r(z)\equiv0, \{z_{j_{0}}, \tilde{c}_{j_{0}}\}_{z_{j_{0}}\in\triangle_{I}^{\pm}}\}$, based on the formula \eqref{q-M1},
the solution of the WKI equation \eqref{WKI-equation} is expressed as
\begin{align}\label{q-Msol}
q^{sol}(x,t,\sigma_{d}^{sol})=q^{sol}(y(x,t),t,\sigma_{d}^{sol})=q_{\pm}+ \frac{(\Phi_{0}+1)^{4}}{(\Phi_{0}+1)^{4}-q_{0}^{4}}\mathcal{M}^{sol}_{z_{j_{0}}}(y,t), y(x,t)=x-c_{-},
\end{align}
\begin{align}\label{c--sol}
c^{sol}_{-}=\lim_{z\rightarrow0}\frac{1}{iz}(&(M^{sol}_{z_{j_{0}}}(y,t,0)^{-1} M^{sol}_{z_{j_{0}}}(y,t,z))_{11}-1)\\ &-\frac{1}{\Phi_{0}+1}(\bar{q}_{\pm}\mathcal{M}^{sol}_{z_{j_{0}}}(y,t) +q_{\pm}\bar{\mathcal{M}}^{sol}_{z_{j_{0}}}(y,t)),
\end{align}
where
\begin{align*}
\mathcal{M}^{sol}_{z_{j_{0}}}(y,t)=\left(\lim_{z\rightarrow0}\frac{\partial}{\partial y}\frac{1}{z}(M^{sol}_{z_{j_{0}}}(y,t,0)^{-1}M^{sol}_{z_{j_{0}}}(y,t,z))_{12}e^{2d} \frac{4\Phi_{0}^{2}}{(\Phi_{0}+1)}\right.\\
 \left.-\frac{q_{\pm}^{2}}{\Phi_{0}+1}\lim_{z\rightarrow0}\frac{\partial}{\partial y}\frac{1}{z}(M^{sol}_{z_{j_{0}}}(y,t,0)^{-1}M^{sol}_{z_{j_{0}}}(y,t,z))_{21}e^{-2d} \frac{4\Phi_{0}^{2}}{(\Phi_{0}+1)}\right).
\end{align*}

\subsection{The error function $\tilde{E}(z)$ between $M^{sol}$ and $M^{sol}_{I}$}

In this subsection, we study the error matrix-function $\tilde{E}(z)$ . We will show that the
error function $\tilde{E}(z)$ solves a small norm RH problem and  can be expanded
asymptotically for large times. According to the definition \eqref{delate-pure-RHP}, we can derive a RH problem related to matrix function $\tilde{E}(z)$.

\begin{RHP}\label{RH-8}
Find a matrix-valued function $\tilde{E}(z)$ satisfies that
\begin{itemize}
 \item $\tilde{E}$ is analytic in $\mathbb{C}\backslash \Sigma^{(2)}$;
 \item $\tilde{E}(z)=I+O(z^{-1})$, \quad $z\rightarrow\infty$;
 \item $\tilde{E}_+(z)=\tilde{E}_-(z)V^{(\tilde{E})}(z)$, \quad $z\in\Sigma^{(2)}$, where
\end{itemize}
 \begin{align}
 V^{(\tilde{E})}(z)=
 M^{sol}_{I}(z)V^{(2)}(z)M^{sol}_{I}(z)^{-1}.
 \end{align}
\end{RHP}

Based on the Proposition \ref{Msol-prop}, $M^{sol}_{I}$ is bounded on $\Sigma^{(2)}$. Then, according to  Proposition \ref{V2-estimate-prop}, we have
\begin{align}\label{VEtilde-estimate}
\big|\big|V^{(\tilde{E})}-\mathbb{I}\big|\big|_{L^p} =\big|\big|V^{(2)}-\mathbb{I}\big|\big|_{L^p} =O(e^{-|t|\frac{2}{\Phi_0}\rho^{2}}),
\end{align}
establishes RHP\ref{RH-8} as a  small-norm RH problem. Therefore, the existence and uniqueness of the solution of the RHP\ref{RH-8} can be shown  by using a small-norm RH problem
\cite{Deift-1994-2,Deift-2003}. Furthermore, based on the Beal-Coifman theory, the solution of RHP \ref{RH-8} is obtained as
\begin{align}\label{Etilde-solution}
\tilde{E}(z)=\mathbb{I}+\frac{1}{2\pi i}\int_{\Sigma^{(2)}\setminus\bigcup_{j=1}^{4}\Sigma_j}\frac{\mu_{\tilde{E}}(s) (V^{(\tilde{E})}(s)-\mathbb{I})}{s-z}\,ds,
\end{align}
where $\mu_{\tilde{E}}\in L^2 (\Sigma^{(2)}\setminus\bigcup_{j=1}^{4}\Sigma_j) $  satisfies
\begin{align}\label{6-10}
(1-C_{\tilde{E}})\mu_{\tilde{E}}=I.
\end{align}
The  integral operator $C_{\tilde{E}}: L^{2}(\Sigma^{(2)}\setminus\bigcup_{j=1}^{4} \Sigma_j\rightarrow\Sigma^{(2)}\setminus\bigcup_{j=1}^{4}\Sigma_j)$ is defined by
\begin{align*}
C_{\tilde{E}}f=C_{-}(f(V^{(\tilde{E})}-\mathbb{I})),\\
C_{-}f(z)\lim_{z\rightarrow\Sigma_{-}^{(2)}} \int_{\Sigma^{(2)}\setminus\bigcup_{j=1}^{4}\Sigma_j}\frac{f(s)}{s-z}ds,
\end{align*}
where $C_{-}$ is the Cauchy projection operator. Based on the properties of the Cauchy projection operator $C_{-}$ and \eqref{VEtilde-estimate}, we obtain that
\begin{align}
\|C_{\mu_{\tilde{E}}}\|_{L^2(\Sigma^{(2)}\setminus\bigcup_{j=1}^{4}\Sigma_j)} \lesssim\|C_-\|_{L^2(\Sigma^{(2)}\setminus\bigcup_{j=1}^{4}\Sigma_j)\rightarrow L^2(\Sigma^{(2)}\setminus\bigcup_{j=1}^{4}\Sigma_j)}\|V^{(\tilde{E})}-I\|_{L^{\infty}
(\Sigma^{(2)}\setminus\bigcup_{j=1}^{4}\Sigma_j)}\lesssim O(e^{-|t|\frac{2}{\Phi_0}\rho^{2}}),
\end{align}
from which we know that $1-C_{\tilde{E}}$ is invertible.
Moreover,
\begin{align}\label{mu-tildeE-estimation}
||\mu_{\tilde{E}}||_{L^{2}(\Sigma^{(2)}\setminus\bigcup_{j=1}^{4}\Sigma_j)}\lesssim \frac{\|C_{\mu_{\tilde{E}}}\|}{1-\|C_{\mu_{\tilde{E}}}\|} O(e^{-|t|\frac{2}{\Phi_0}\rho^{2}}).
\end{align}
As a result, the existence and uniqueness of  $\tilde{E}(z)$ is guaranteed. 

Next, in order to reconstruct the solutions of the WKI equation \eqref{WKI-equation}, it is necessary to study the asymptotic behavior of $\tilde{E}(z)$ as $z\rightarrow\infty$ and large time asymptotic behavior of $\tilde{E}(0)$.

\begin{prop}
The $\tilde{E}(z)$ defined in \eqref{Etilde-define} satisfies
\begin{align}\label{Etilde-estimation}
|\tilde{E}(z)-\mathbb{I}|\lesssim O(e^{-|t|\frac{2}{\Phi_0}\rho^{2}}).
\end{align}
When $z=0$,
\begin{align*}
 \tilde{E}(0)=\mathbb{I}+\frac{1}{2\pi i}\int_{\Sigma^{(2)}\setminus\bigcup_{j=1}^{4}\Sigma_j}\frac{\mathbb{I}+\mu_{\tilde{E}}(s) (V^{(\tilde{E})}(s)-\mathbb{I})}{s}\,ds.
\end{align*}
As $z$ approaches to zero, $\tilde{E}(z)$ possesses expansion at $z=0$,
\begin{align}
\tilde{E}(z)=\tilde{E}(0)+\tilde{E}_{1}z+O(z^{2}),
\end{align}
where
\begin{align*}
 \tilde{E}_{1}=\frac{1}{2\pi i}\int_{\Sigma^{(2)}\setminus\bigcup_{j=1}^{4}\Sigma_j} \frac{(\mathbb{I}+\mu_{\tilde{E}(s)})(V^{(\tilde{E})}(s)-I)}{s^{2}}\,ds.
\end{align*}
Additionally, $\tilde{E}(0)$ and $\tilde{E}_{1}$ satisfy the following estimation as $t\rightarrow\infty$,
\begin{align}\label{tildeE0-E1-estimation}
|\tilde{E}(0)-\mathbb{I}|\lesssim O(e^{-|t|\frac{2}{\Phi_0}\rho^{2}}),~~ |\tilde{E}_{1}|\lesssim O(e^{-|t|\frac{2}{\Phi_0}\rho^{2}}).
\end{align}
\end{prop}
\begin{proof}
Based on \eqref{VEtilde-estimate} and \eqref{mu-tildeE-estimation} , the formula \eqref{Etilde-estimation} can be derived directly. Furthermore, by observing that $|s|^{-1}$ and $|s|^{-2}$ are bounded on $\Sigma^{(2)}\setminus\bigcup_{j=1}^{4}\Sigma_j$, it is not hard to obtain the formulae in \eqref{tildeE0-E1-estimation}.
\end{proof}

\section{Local solvable model near phase point $z=z_0$}\label{section-Local-solvable-model}

For $z\in U_{z_{0}}$, the Proposition \ref{V2-estimate-prop} implies that $V^{(2)}-I$ does not have a uniform estimate for large time. In order to deal with this case, we introduce  a  model $M^{sol}(z)M^{(pc)}(z_0,r_0)$ to match the the jumps of $M^{(2)}_{R}$ on $\Sigma^{(2)}\cap U_{z_{0}}$, and establish a local solvable model  for the error function $E(z)$. According to the fact that $\theta(z)=\left(\frac{zy}{t}+\frac{z^{2}}{\Phi_{0}}\right)$ and $z_{0}=-\left(\frac{y}{2t}\Phi_{0}\right)$, we introduce the transformation
\begin{align}
\lambda=\lambda(z)=i\sqrt{\frac{2t}{\Phi_0}}(z-z_{0}).
\end{align}
Then, we can derive that
\begin{align}
t\theta=-\frac{1}{2}\lambda^{2}-\frac{t}{\Phi_0}z_{0}^{2},
\end{align}
from which we know that  $U_{z_0}$ is mapped into an expanding neighborhood of $\lambda=0$. If we let
\begin{align}
r_0(z_0)=-r(z_0)T_0^{-2}(z_0) e^{2i(\nu(z_0)log(\sqrt{\frac{2t}{\Phi_0}}))}e^{\frac{it}{\Phi_0}z_0^{2}},\label{r-match}
\end{align}
Furthermore, by considering the fact that $1-\chi_{\mathcal{Z}}=1$ as $z\in U_{z_0}$, the jump of $M^{(2)}_{R}$ in $U_{z_0}$ is translated into
\begin{align}\label{6-7}
V^{(2)}(z)\mid_{z\in U_{z_0}}=\left\{\begin{aligned}
\lambda(z)^{-i\nu\hat{\sigma}_{3}}e^{\frac{i\lambda(z)^{2}}{4}
\hat{\sigma}_{3}}\left(
                    \begin{array}{cc}
                      1 & r_{0}(z_0) \\
                      0 & 1 \\
                    \end{array}
                  \right),\quad z\in\Sigma_{1},\\
\lambda(z)^{-i\nu\hat{\sigma}_{3}}e^{\frac{i\lambda(z)^{2}}{4}
\hat{\sigma}_{3}}\left(
                    \begin{array}{cc}
                      1 & 0 \\
                      \frac{\bar{r}_{0}(z_0)}{1+|r_{0}(z_0)|^{2}} & 1 \\
                    \end{array}
                  \right),\quad z\in\Sigma_{2},\\
\lambda(z)^{-i\nu\hat{\sigma}_{3}}e^{\frac{i\lambda(z)^{2}}{4}
\hat{\sigma}_{3}}\left(
                    \begin{array}{cc}
                      1 & \frac{r_{0}(z_0)}{1+|r_{0}(z_0)|^{2}}\\
                      0 & 1 \\
                    \end{array}
                  \right),\quad z\in\Sigma_{3},\\
\lambda(z)^{-i\nu\hat{\sigma}_{3}}e^{\frac{i\lambda(z)^{2}}{4}
\hat{\sigma}_{3}}\left(
                    \begin{array}{cc}
                      1 & 0 \\
                      \bar{r}_{0}(z_0) & 1 \\
                    \end{array}
                  \right),\quad z\in\Sigma_{4}.
\end{aligned}\right.
\end{align}
Obviously, the jump $V^{(2)}(z)\mid_{z\in U_{z_0}}$ in \eqref{6-7} is equivalent to the jump of the parabolic cylinder model problem \eqref{Vpc} whose solutions is shown in Appendix $A$. Additionally, since  $M^{sol}(z)$ is an analytic and bounded function in  $U_{z_0}$, and  $M^{(2)}_{R}(z)=M^{sol}(z)M^{(pc)}(z_0,r_0)(z\in\mathcal{U}_{z_0})$, we can derive that that $M^{sol}(z)M^{(pc)}(z_0,r_0)$ satisfies the jump $V^{(2)}(z)$ of $M^{(2)}_{R}(z)$ via a direct calculation.

\section{The small norm RHP for error function $E(z)$}\label{small-norm-RHP-E}

According to \eqref{Mrhp},  the unknown error function $E(z)$ can be shown as
\begin{align}\label{explict-E(z)}
E(z)=\left\{\begin{aligned}
&M^{(2)}_{R}(z)M^{sol}(z)^{-1}, &&z\in\mathbb{C}\backslash U_{z_0},\\
&M^{(2)}_{R}(z)M^{(pc)}(z_0,r_0)^{-1}M^{sol}(z)^{-1}, &&z\in U_{z_0}.
\end{aligned} \right.
\end{align}
Obviously, $E(z)$ is analytic in $\mathbb{C}\setminus\Sigma^{(E)}$
where
\begin{align}
\Sigma^{(E)}=\partial U_{z_0}\cup(\Sigma^{(2)}\backslash U_{z_0}),
\end{align}
with clockwise direction for $\partial U_{z_0}$.
Then, we obtian a Riemann-Hilbert problem for error function $E(z)$.
\begin{RHP}\label{RH-9}
Find a matrix-valued function $E(z)$ satisfies that
\begin{itemize}
 \item $E(z)$ is analytic in $\mathbb{C}\backslash \Sigma^{(E)}$;
 \item $E(z)=I+O(z^{-1})$, \quad $z\rightarrow\infty$;
 \item $E_+(z)=E_-(z)V^{(E)}(z)$, \quad $z\in\Sigma^{(E)}$, where
\end{itemize}
 \begin{align}\label{6-8}
 V^{(E)}(z)=\left\{\begin{aligned}
 &M^{sol}(z)V^{(2)}(z)M^{sol}(z)^{-1}, &&z\in\Sigma^{(2)}\backslash U_{z_0},\\
 &M^{sol}(z)M^{(pc)}(z_0,r_0)M^{sol}(z)^{-1}, &&z\in\partial U_{z_0}.
 \end{aligned}\right.
 \end{align}
\end{RHP}

\centerline{\begin{tikzpicture}[scale=0.7]
\draw[blue,-][dashed](-6,0)--(6,0);
\draw[-][thick](-6,-1)--(6,1);
\draw[-][thick](-6,1)--(6,-1);
\draw[fill][white] (0,0) circle [radius=2];
\draw[->][thick](3,0.5)--(4.5,0.75);
\draw[->][thick](3,-0.5)--(4.5,-0.75);
\draw[->][thick](-4.5,-0.75)--(-3,-0.5);
\draw[->][thick](-4.5,0.75)--(-3,0.5);
\draw[fill] (3,1.5)node[left]{$\partial U_{z_{0}}$};
\draw[fill] (5.8,1)node[above]{$\Sigma^{(2)}\setminus U_{z_{0}}$};
\draw(0,0) [black, line width=1] circle(2);
\draw[->][black, line width=1] (2,0) arc(0:-270:2);
\end{tikzpicture}}
\centerline{\noindent {\small \textbf{Figure 4.} Jump contour $\Sigma^{(E)}=\partial U_{z_{0}}\cup(\Sigma^{(2)}\setminus U_{z_{0}})$}.}

On the basis of Proposition \ref{V2-estimate-prop}, the boundedness of $M^{sol}$ and \eqref{A-1}, we can obtian the following estimates  immediately.
\begin{align}\label{VE-I-estimate}
|V^{(E)}(z)-\mathbb{I}|=\left\{\begin{aligned}
&O(e^{-|t|\frac{\varepsilon}{4\Phi_0}\rho^2}) &&z\in \left(\bigcup_{j=1}^{4}\Sigma_j\right)\backslash U_{z_0},\\
&O(|t|^{-1/2}) &&z\in\partial U_{z_0}.
\end{aligned}\right.
\end{align}
Moreover, we have
\begin{align}\label{z-z-VE}
\big|\big|(z-z_0)^{k}(V^{(E)}-\mathbb{I})\big|\big|_{L^{p}(\Sigma^{(E)})}=o(t^{-1/2}),~~p\in[1,+\infty],~k\geq0.
\end{align}
The estimates \eqref{VE-I-estimate} mean that the bound on $V^{(E)}(z)-\mathbb{I}$ decay uniformly. Therefore, by using a small-norm Riemann-Hilbert problem, the existence and uniqueness of the  solution of RHP \ref{RH-9} can be expressed \cite{Deift-2003,Deift-1994-2}. Moreover, based on the Beal-Coifman theory, the solution of RHP \ref{RH-9} is obtained as
\begin{align}\label{Ez-solution}
E(z)=\mathbb{I}+\frac{1}{2\pi i}\int_{\Sigma^{(E)}}\frac{(\mathbb{I}+\mu_E(s))(V^{(E)}(s)-I)}{s-z}ds,
\end{align}
where $\mu_E\in L^2 (\Sigma^{(E)}) $  is the unique solution of
\begin{align}
(1-C_{\omega_E})\mu_E=\mathbb{I}.
\end{align}
The  integral operator $C_{\omega_E}(L^{\infty}(\Sigma^{(E)})\rightarrow L^{2}(\Sigma^{(E)}))$ is defined by
\begin{align*}
C_{\omega_E}f(z)=C_{-}(f(V^{(E)}(z)-\mathbb{I})),\\
C_{-}(f)(z)\lim_{z\rightarrow\Sigma_{-}^{(E)}}\int_{\Sigma^{(E)}}\frac{f(s)}{s-z}ds,
\end{align*}
where $C_{-}$ is the Cauchy projection operator. On the basis of the properties of the Cauchy projection operator $C_{-}$ and \eqref{z-z-VE}, we obtain that
\begin{align}
\|C_{\omega_E}\|_{L^2(\Sigma^{(E)})}\lesssim\|C_-\|_{L^2(\Sigma^{(E)})\rightarrow L^2(\Sigma^{(E)})}\|V^{(E)}-\mathbb{I}\|_{L^{\infty}
(\Sigma^{(E)})}\lesssim O(|t|^{-1/2}),
\end{align}
which implies that $1-C_{\omega_E}$ is invertible. As a result, for large time $t$,
\begin{align}
\|\mu_E(s)\|\lesssim \frac{\|C_{\omega_E}\|}{1+\|C_{\omega_E}\|}\lesssim O(|t|^{-1/2}).
\end{align}
Then the existence and uniqueness of  $E(z)$ is guaranteed. This  explains that it is reasonable to define $M^{(2)}_{R}$ in \eqref{Mrhp}. In turn we can solve \eqref{delate-pure-RHP} to the unknown $M^{(3)}$ which satisfies the Riemann-Hilbert Problem \ref{RH-5}.

Furthermore, in order to reconstruct the solutions of $q(y,t)$, it is necessary to study the asymptotic behavior of $E(z)$ as $z\rightarrow0$ and large time asymptotic behavior of $E(0)$. By observing  the estimate \eqref{VE-I-estimate},  for $t\rightarrow-\infty$, we just need to consider the calculation on $\partial U_{z_{0}}$ since it approaches to zero exponentially on other boundary. Firstly, as $z\rightarrow 0$, we expand $E(z)$ as
\begin{align}
E(z)=E(0)+E_{1}z+O(z^{2}),
\end{align}
where
\begin{gather}
E(0)=\mathbb{I}+\frac{1}{2\pi i}\int_{\Sigma^{(E)}}\frac{(\mathbb{I}+\mu_E(s))(V^{(E)}(s)-\mathbb{I})}{s}\,ds,\label{6-12}\\
E_{1}=-\frac{1}{2\pi i}\int_{\Sigma^{(E)}}\frac{(\mathbb{I}+\mu_E(s))(V^{(E)}(s)-I)}{s^{2}}\,ds.\label{6-13}
\end{gather}
Then, as $t\rightarrow-\infty$, the asymptotic behavior  of $E(0)$ and $E_{1}$  can be calculated as
\begin{align}
E(0)=&\mathbb{I}+\frac{1}{2i\pi}\int_{\partial U_{ z_{0}}}(V^{(E)}(s)-\mathbb{I})ds+o(|t|^{-1})\notag\\
=&\mathbb{I}+\frac{\sqrt{\Phi_0}}{2iz_0\sqrt{2t}}M^{sol}(z_0)^{-1} M_1^{(pc,-)}(z_0)M^{sol}(z_0),\label{6-14}\\
E_{1}=&-\frac{\sqrt{\Phi_0}}{2iz^{2}_{0}\sqrt{2t}}M^{sol}(z_0)^{-1}M_1^{(pc,-)}(z_0)M^{sol}(z_0),\label{6-15}
\end{align}
where $M_1^{(pc,-)}=-i\left(
                       \begin{array}{cc}
                         0 & \beta^{-}_{12}(r_{0}) \\
                         \beta^{-}_{21}(r_{0}) & 0 \\
                       \end{array}
                     \right)$.
Then, via using \eqref{A-3} and \eqref{r-match}, we have
\begin{align*}
\beta^{-}_{12}(r(z_{0}))=\overline{\beta^{-}}_{21}(r(z_{0}))=\alpha(z_{0},-)e^{i\frac{y^{2}}{4t}\Phi_0+i\nu(z_{0})\log\frac{2}{\Phi_0}|t|},
\end{align*}
where
\begin{align*}
&|\alpha(z_{0},-)|^{2}=|\nu(z_{0})|,\\
&\arg\alpha(z_{0},-)=\frac{\pi}{4}-\arg\Gamma(i\nu(z_{0}))-\arg r(z_{0})-4\mathop{\sum}\limits_{k\in\triangle_{z_{0}}^{+}}\arg(z-z_{k})- 2\int_{-\infty}^{z_{0}}\log|z_{0}-s|d\nu(s).
\end{align*}
Moreover, from \eqref{6-14}, a direct calculation shows that
\begin{align}\label{6-16}
E(0)^{-1}=\mathbb{I}+O(t^{-1/2}).
\end{align}

\section{Analysis on pure $\bar{\partial}$-Problem}\label{section-Pure-dbar-RH}

In this section, we are going to investigate the existence and asymptotic behavior of the remaining $\bar{\partial}$-problem $M^{(3)}(z)$.  The pure $\bar{\partial}$-RH problem \ref{RH-5} for $M^{(3)}(z)$ is equivalent to the following integral equation
\begin{align}\label{7-1}
M^{(3)}(z)=\mathbb{I}-\frac{1}{\pi}\iint_{\mathbb{C}}\frac{M^{(3)}W^{(3)}}{s-z}\mathrm{d}A(s),
\end{align}
where $\mathrm{d}A(s)$ is Lebesgue measure. Next, we rewrite the integral equation \eqref{7-1} as the following   operator form
\begin{align}\label{7-2}
(\mathbb{I}-\mathrm{S})M^{(3)}(z)=\mathbb{I},
\end{align}
where $\mathrm{S}$ is Cauchy operator
\begin{align}\label{7-3}
\mathrm{S}[f](z)=\frac{1}{\pi}\iint_{\mathbb{C}}\frac{f(s)W^{(3)}(s)}{s-z}\mathrm{d}A(s).
\end{align}
The euqation \eqref{7-2} implies that if the inverse operator $(\mathbb{I}-\mathrm{S})^{-1}$ exists, the solution $M^{(3)}(z)$  also exists. In order to prove the operator $\mathbb{I}-\mathrm{S}$ is reversible, we give the following proposition.

\begin{prop}
For sufficiently large $|t|$, there exists a constant $c$ that enables the operator \eqref{7-3} to admit the following relation
\begin{align}\label{7-4}
||\mathrm{S}||_{L^{\infty}\rightarrow L^{\infty}}\leq c|t|^{-1/4}.
\end{align}
\end{prop}

\begin{proof}
We mainly pay attention to  the case that the matrix function supported in the region $\Omega_1$, the other cases can be proved similarly. Make the assumption  that $f\in L^{\infty}(\Omega_1)$ and $s=p+iq$. Then based on \eqref{4.1}, \eqref{5-1} and the boundedness of $M^{(2)}_{R}$, we obtain the following inequality
\begin{align}\label{7-5}
|S[f](z)|&\leq\frac{1}{\pi}\iint_{\Omega_{1}}\frac{|f(s)M^{(2)}_{R}(s)\bar{\partial}R_{1}(s)M^{(2)}_{R}(s)^{-1}|}
{|s-z|}df(s)\notag\\
&\leq c\frac{1}{\pi}\iint_{\Omega_{1}}
\frac{|\bar{\partial}R_{1}(s)||e^{\frac{2t}{\Phi_0}q(p-z_{0})}|}{|s-z|}df(s),
\end{align}
where $c$ is a constant. Then, on the basis of \eqref{R-estimate} and the estimates demonstrated in Appendix $B$, we obtain the following norm estimate.
\begin{align}\label{7-6}
||\mathrm{S}||_{L^{\infty}\rightarrow L^{\infty}}\leq c(I_{1}+I_{2}+I_{3})\leq ct^{-1/4},
\end{align}
where
\begin{align}\label{7-8}
I_{1}=\iint_{\Omega_{1}}
\frac{|\bar{\partial}\chi_{\mathcal{Z}}(s)||e^{\frac{2t}{\Phi_0}q(p-z_{0})}|}{|s-z|}df(s), ~~
I_{2}=\iint_{\Omega_{1}}
\frac{|r'(p)||e^{\frac{2t}{\Phi_0}q(p-z_{0})}|}{|s-z|}df(s),
\end{align}
and
\begin{align}\label{7-9}
I_{3}=\iint_{\Omega_{1}}
\frac{|s-z_{0}|^{-\frac{1}{2}}|e^{\frac{2t}{\Phi_0}q(p-z_{0})}|}{|s-z|}df(s).
\end{align}
\end{proof}

Next, in order  to reconstruct the potential $q(x,t)$ as $|t|\rightarrow\infty$,  based on \eqref{q-sol}, we need  to investigate the large time asymptotic behaviors of  $M^{(3)}(0)$  and $M_{1}^{(3)}(y,t)$ which  are defined in the asymptotic expansion of $M^{(3)}(z)$ as $z\rightarrow0$, i.e.,
\begin{align*}
M^{(3)}(z)=M^{(3)}(0)+M_{1}^{(3)}(y,t)z+O(z^{2}),~~z\rightarrow0 ,
\end{align*}
where
\begin{align*}
M^{(3)}(0)=\mathbb{I}-\frac{1}{\pi}\iint_{\mathbb{C}}\frac{M^{(3)}(s)W^{(3)}(s)}{s}
\mathrm{d}A(s),\\
M^{(3)}_{1}(y,t)=\frac{1}{\pi}\int_{\mathbb{C}}\frac{M^{(3)}(s)W^{(3)}(s)}{s^{2}}
\mathrm{d}A(s).
\end{align*}
Then, $M^{(3)}(0)$ and $M^{(3)}_{1}(y,t)$ satisfy the following proposition.
\begin{lem}\label{prop-M3-Est}
For $t\rightarrow-\infty$,  $M^{(3)}(0)$ and $M^{(3)}_{1}(y,t)$ satisfy  the following inequality
\begin{align}
\|M^{(3)}(0)-\mathbb{I}\|_{L^{\infty}}\lesssim |t|^{-\frac{3}{4}},\label{8-10}\\
M^{(3)}_{1}(y,t)\lesssim |t|^{-\frac{3}{4}}\label{8-11}.
\end{align}
\end{lem}
The proof of this Proposition is similar to the process that shown in  Appendix $B$.

\section{Asymptotic approximation for the WKI equation}

Now, we are ready to construct the long time asymptotic of the WKI equation.
Inverting a series of transformation including \eqref{Trans-1}, \eqref{Trans-2}, \eqref{delate-pure-RHP} and \eqref{Mrhp}, i.e.,
\begin{align*}
M(z)\leftrightarrows M^{(1)}(z)\leftrightarrows M^{(2)}(z)\leftrightarrows M^{(3)}(z) \leftrightarrows E(z),
\end{align*}
we then have
\begin{align*}
M(z)=M^{(3)}(z)E(z)M^{sol}(z)R^{(2)^{-1}}(z)T^{-\sigma_{3}}(z),~~ z\in\mathbb{C}\setminus U_{z_{0}}.
\end{align*}
In order to recover the potential $q(x,t)$ , we take $z\rightarrow0$ along the imaginary axis, as a result $R^{(2)}(z)=I$. Then, we obtain
\begin{gather*}
M(0)=M^{(3)}(0)E(0)M^{sol}(0)T^{-\sigma_{3}}(0),\\
M=\left(M^{(3)}(0)+M^{(3)}_{1}z+\cdots\right)\left(E(0)+E_{1}z+\cdots\right)
\left(M^{sol}(z)\right)\left(T^{-\sigma_{3}}(0)
+T_{1}^{-\sigma_{3}}z+\cdots\right).
\end{gather*}
Then by simple calculation, we immediately obtain
\begin{align*}
M(0)^{-1}M(z)=&T^{\sigma_{3}}(0)M^{sol}(0)^{-1}M^{sol}(z)T^{-\sigma_{3}}(0)z\\
&+ T^{\sigma_{3}}(0)M^{sol}(0)^{-1}E_{1}M^{sol}(z)T^{-\sigma_{3}}(0)z\\
&+ T^{\sigma_{3}}(0)M^{sol}(0)^{-1}M^{sol}(z)T^{-\sigma_{3}}_{1}z+O(t^{-\frac{3}{4}}).
\end{align*}
Then, according to the reconstruction formula \eqref{q-sol}, \eqref{q-Msol} and Proposition \ref{Msol-prop}, as $t\rightarrow-\infty$, we obtain that
\begin{align}\label{9.1}
q(x,t)&=q(y(x,t),t)\notag\\
&=q^{sol}(y(x,t),t;\sigma^{sol}_{d})T^{2}(0)(1+T_{1})-it^{-\frac{1}{2}} \frac{(\Phi_{0}+1)^{4}}{(\Phi_{0}+1)^{4}-q_{0}^{4}}F(y,t) +O(t^{-\frac{3}{4}}),\\
F(y,t)&=\left(\frac{\partial}{\partial y}f^{-}_{12}e^{2d} \frac{4\Phi_{0}^{2}}{(\Phi_{0}+1)}
-\frac{q_{\pm}^{2}}{\Phi_{0}+1}\frac{\partial}{\partial y}f^{-}_{21}e^{-2d} \frac{4\Phi_{0}^{2}}{(\Phi_{0}+1)}\right),
\end{align}

where
\begin{align*}
y(x,t)&=x-c^{sol}_{-}(x,t,\hat{\sigma}_{d}(I))-iT_{1}^{-1} -it^{-\frac{1}{2}}f^{-}_{11}-\frac{1}{\Phi_{0}+1}(\bar{q}_{\pm}F(y,t) +q_{\pm}F(y,t))+O(t^{-\frac{3}{4}}),\\
f^{-}_{12}&=\frac{\sqrt{\Phi_0}}{2iz^{2}_{0}\sqrt{2}}
[M^{sol}(0)^{-1}(M^{sol}(z_0)^{-1}M_{1}^{(pc,-)}(z_{0})M^{sol}(z_0)]_{12},\\
f^{-}_{21}&=\frac{\sqrt{\Phi_0}}{2iz^{2}_{0}\sqrt{2}}
[M^{sol}(0)^{-1}(M^{sol}(z_0)^{-1}M_{1}^{(pc,-)}(z_{0})M^{sol}(z_0)]_{21},\\
f^{-}_{11}&=\frac{\sqrt{\Phi_0}}{2iz^{2}_{0}\sqrt{2}}
[M^{sol}(0)^{-1}(M^{sol}(z_0)^{-1}M_{1}^{(pc,-)}(z_{0})M^{sol}(z_0)]_{11},
\end{align*}
where $M_{1}^{(pc,-)}$ is defined in section \ref{small-norm-RHP-E}.

For the initial value problem of the WKI equation i.e., $q_{0}(x)\in \mathcal{H}(\mathbb{R})$, the long time asymptotic behavior \eqref{9.1} gives the soliton resolution  which contains the soliton term confirmed by $N(I)$-soliton on discrete spectrum and the $t^{-\frac{1}{2}}$ order term on continuous spectrum with residual error up to $O(t^{-\frac{3}{4}})$. Additionally, our results reveal that the soliton solutions of WKI equation are asymptotic stable.

\begin{rem}
The steps in the steepest descent analysis of RHP \ref{RH-2} for $t\rightarrow+\infty$ is similar to the case  $t\rightarrow-\infty$ which has been presented in section $5$-$9$. When we consider $t\rightarrow+\infty$, the main difference can be traced back to the fact that the regions of growth and decay of the exponential factors $e^{2it\theta}$ are reversed, see Fig. 1. Here, we leave the detailed calculations to the interested reader.
\end{rem}

Finally, we can give the results shown in the following Theorem.

\begin{thm}\label{Thm-1}
Suppose that the initial value $q_{0}(x)$ satisfies \eqref{boundary} and the Assumption \eqref{assum}. Let $q(x,t)$ be the solution of WKI equation \eqref{WKI-equation}. The scattering data is denoted as $\{r,\{z_{k},c_{k}\}_{k=1}^{N}\}$ which generated from the initial value $q_{0}(x)$. Let $z_0=-\frac{y\Phi_0}{2t}$ and denote $q^{sol}(y(x,t),t;\sigma^{sol}_{d})$ be the $N(I)$-soliton solution corresponding to scattering data $\sigma_{d}^{sol}=\{0, \{z_{j_{0}}, \tilde{c}_{j_{0}}\}_{z_{j_{0}}\in\triangle_{I}^{\pm}}\}$ where $I$ is defined in \eqref{rho-I-definition}, the solution $q(x,t)$ can be expressed as
\begin{align}\label{9.2}
\begin{split}
q(x,t)&=q(y(x,t),t)\\
&=q^{sol}(y(x,t),t;\sigma^{sol}_{d})T^{2}(0)(1+T_{1})-it^{-\frac{1}{2}} \frac{(\Phi_{0}+1)^{4}}{(\Phi_{0}+1)^{4}-q_{0}^{4}}F(y,t) +O(t^{-\frac{3}{4}}),\\
F(y,t)&=\left(\frac{\partial}{\partial y}f^{\pm}_{12}e^{2d} \frac{4\Phi_{0}^{2}}{(\Phi_{0}+1)}
-\frac{q_{\pm}^{2}}{\Phi_{0}+1}\frac{\partial}{\partial y}f^{\pm}_{21}e^{-2d} \frac{4\Phi_{0}^{2}}{(\Phi_{0}+1)}\right),
\end{split}
\end{align}
where $q^{sol}(y(x,t),t;\sigma^{sol}_{d})$ is defined in \eqref{q-Msol},
\begin{align*}
y(x,t)&=x-c^{sol}_{-}(x,t,\hat{\sigma}_{d}(I))-iT_{1}^{-1} -it^{-\frac{1}{2}}f^{-}_{11}-\frac{1}{\Phi_{0}+1}(\bar{q}_{\pm}F(y,t) +q_{\pm}F(y,t))+O(t^{-\frac{3}{4}}),\\
f^{\pm}_{12}&=\frac{\sqrt{\Phi_0}}{2iz^{2}_{0}\sqrt{2}}
[M^{sol}(0)^{-1}(M^{sol}(z_0)^{-1}M_{1}^{(pc,\pm)}(z_{0})M^{sol}(z_0)]_{12},\\
f^{\pm}_{21}&=\frac{\sqrt{\Phi_0}}{2iz^{2}_{0}\sqrt{2}}
[M^{sol}(0)^{-1}(M^{sol}(z_0)^{-1}M_{1}^{(pc,\pm)}(z_{0})M^{sol}(z_0)]_{21},\\
f^{\pm}_{11}&=\frac{\sqrt{\Phi_0}}{2iz^{2}_{0}\sqrt{2}}
[M^{sol}(0)^{-1}(M^{sol}(z_0)^{-1}M_{1}^{(pc,\pm)}(z_{0})M^{sol}(z_0)]_{11},
\end{align*}
and $c_{-}(x,t,\sigma^{sol}_{d})$ is defined in \eqref{c--sol}, $T(z)$ is defined in \eqref{T-define}, $T_{1}$ is defined in Proposition \ref{T-property}, $d$ is defined in \eqref{defin-d}
Here, $M_{1}^{pc,\pm}(z)$ can be expressed as
$M_1^{(pc,\pm)}=-i\left(
                       \begin{array}{cc}
                         0 & \beta^{\pm}_{12}(r_{0}) \\
                         \beta^{\pm}_{21}(r_{0}) & 0 \\
                       \end{array}
                     \right)$, and $ M^{sol}(z)$ is defined in \eqref{Mrhp}.
Here,
\begin{align*}
\beta^{\pm}_{12}(r(z_{0}))=\overline{\beta_{21}^{\pm}}(r(z_{0}))=\alpha(z_{0},\pm) e^{i\frac{y^{2}}{4t}\Phi_0+i\nu(z_{0})\log\frac{2}{\Phi_0}|t|},
\end{align*}
where
\begin{align*}
&|\alpha(z_{0},\pm)|^{2}=|\nu(z_{0})|,\\
&\arg\alpha(z_{0},\pm)=\pm\frac{\pi}{4}\mp\arg\Gamma(i\nu(z_{0}))-\arg r(z_{0})-4\mathop{\sum}\limits_{k\in\triangle_{z_{0}}^{+}} \mp2\int_{-\infty}^{z_{0}}\log|z_{0}-s|d\nu(s),
\end{align*}
$r(z)$ is defined in \eqref{r-expression}, $\nu(z)$ is defined in \eqref{delta-v-define} and  $\Gamma$ denotes the gamma function.
\end{thm}

\noindent{\bf Acknowledgments}
This work was supported by  the National Natural Science Foundation of China under Grant No. 11975306, the Natural Science Foundation of Jiangsu Province under Grant No. BK20181351, the Six Talent Peaks Project in Jiangsu Province under Grant No. JY-059,  and the Fundamental Research Fund for the Central Universities under the Grant Nos. 2019ZDPY07 and 2019QNA35.
\\

\section*{Appendix A: The parabolic cylinder model problem}
Here, we describe the solution of  parabolic cylinder model problem\cite{PC-model,PC-model-2}.
Define the contours $\Sigma^{pc}=\cup_{j=1}^{4}\Sigma_{j}^{pc}$ where
\begin{align}
\begin{split}
\Sigma_{1}^{pc}=\left\{\lambda\in\mathbb{C}|\lambda=\mathbb{R}^{+}e^{\phi i} \right\},~~ \Sigma_{2}^{pc}=\left\{\lambda\in\mathbb{C}|\lambda=\mathbb{R}^{+}e^{(\pi-\phi) i} \right\},\\ \Sigma_{3}^{pc}=\left\{\lambda\in\mathbb{C}|\lambda=\mathbb{R}^{+}e^{(-\pi+\phi) i} \right\},~~ \Sigma_{4}^{pc}=\left\{\lambda\in\mathbb{C}|\lambda=\mathbb{R}^{+}e^{-\phi i} \right\}.
\end{split}\tag{A.1}
\end{align}
For $r_{0}\in \mathbb{C}$, let $\nu(r)=-\frac{1}{2\pi}\log(1+|r_{0}|^{2})$, consider the following parabolic cylinder model Riemann-Hilbert problem.
\begin{RHP}\label{PC-model}
Find a matrix-valued function $M^{(pc)}(\lambda)$ such that
\begin{align}
&\bullet \quad M^{(pc)}(\lambda)~ \text{is analytic in}~ \mathbb{C}\setminus\Sigma^{pc}, \tag{A.2}\\
&\bullet \quad M_{+}^{(pc)}(\lambda)=M_{-}^{(pc)}(\lambda)V^{(pc)}(\lambda),\quad
\lambda\in\Sigma^{pc}, \tag{A.3}\\
&\bullet \quad M^{(pc)}(\lambda)=I+\frac{M_{1}}{\lambda}+O(\lambda^{2}),\quad
\lambda\rightarrow\infty, \tag{A.4}
\end{align}
where
\begin{align}\label{Vpc}
V^{(pc)}(\lambda)=\left\{\begin{aligned}
\lambda^{-i\nu\hat{\sigma}_{3}}e^{\frac{i\lambda^{2}}{4}
\hat{\sigma}_{3}}\left(
                    \begin{array}{cc}
                      1 & r_{0} \\
                      0 & 1 \\
                    \end{array}
                  \right),\quad \lambda\in\Sigma_{1}^{pc},\\
\lambda^{-i\nu\hat{\sigma}_{3}}e^{\frac{i\lambda^{2}}{4}
\hat{\sigma}_{3}}\left(
                    \begin{array}{cc}
                      1 & 0 \\
                      \frac{\bar{r}_{0}}{1+|r_{0}|^{2}} & 1 \\
                    \end{array}
                  \right),\quad \lambda\in\Sigma_{2}^{pc},\\
\lambda^{-i\nu\hat{\sigma}_{3}}e^{\frac{i\lambda^{2}}{4}
\hat{\sigma}_{3}}\left(
                    \begin{array}{cc}
                      1 &\frac{r_{0}}{1+|r_{0}|^{2}} \\
                      0 & 1 \\
                    \end{array}
                  \right),\quad \lambda\in\Sigma_{3}^{pc},\\
\lambda^{-i\nu\hat{\sigma}_{3}}e^{\frac{i\lambda^{2}}{4}
\hat{\sigma}_{3}}\left(
                    \begin{array}{cc}
                      1 & 0 \\
                      \bar{r}_{0} & 1 \\
                    \end{array}
                  \right),\quad \lambda\in\Sigma_{4}^{pc}.
\end{aligned}\right.\tag{A.5}
\end{align}
\end{RHP}
\centerline{\begin{tikzpicture}[scale=0.7]
\draw[pink,-][dashed](-6,0)--(6,0);
\draw[->][thick](1.5,0.25)--(3,0.5);
\draw[->][thick](-3,0.5)--(-1.5,0.25);
\draw[->][thick](-3,-0.5)--(-1.5,-0.25);
\draw[->][thick](1.5,-0.25)--(3,-0.5);
\draw[blue,-][thick](-6,-1)--(6,1);
\draw[blue,-][thick](-6,1)--(6,-1);
\draw[fill] (6,1)node[right]{$\Sigma_{1}^{pc}$};
\draw[fill] (6,-1)node[right]{$\Sigma_{4}^{pc}$};
\draw[fill] (-6,1)node[left]{$\Sigma_{2}^{pc}$};
\draw[fill] (-6,-1)node[left]{$\Sigma_{3}^{pc}$};
\draw[fill] (0,0)node[below]{$0$};
\draw[fill] (5.8,0)node[below]{$\Omega_{4}$};
\draw[fill] (5.8,0)node[above]{$\Omega_{1}$};
\draw[fill] (-5.8,0)node[below]{$\Omega_{3}$};
\draw[fill] (-5.8,0)node[above]{$\Omega_{2}$};
\draw[fill] (7,3.5)node[below]{$\lambda^{-i\nu\hat{\sigma}_{3}}e^{\frac{i\lambda^{2}}{4}\hat{\sigma}_{3}}
\left(
  \begin{array}{cc}
    1 & r_{0} \\
    0 & 1 \\
  \end{array}
\right)
$};
\draw[fill] (7,-1.5)node[below]{$\lambda^{-i\nu\hat{\sigma}_{3}}e^{+\frac{i\lambda^{2}}{4}\hat{\sigma}_{3}}
\left(
  \begin{array}{cc}
    1 & 0\\
    \bar{r}_{0} & 1 \\
  \end{array}
\right)
$};
\draw[fill] (-7,3.5)node[below]{$\lambda^{-i\nu\hat{\sigma}_{3}}e^{\frac{i\lambda^{2}}{4}\hat{\sigma}_{3}}
\left(
  \begin{array}{cc}
    1 & 0 \\
   \frac{\bar{r}_{0}}{1+|r_{0}|^{2}} & 1 \\
  \end{array}
\right)
$};
\draw[fill] (-7,-1.5)node[below]{$\lambda^{-i\nu\hat{\sigma}_{3}}e^{\frac{i\lambda^{2}}{4}\hat{\sigma}_{3}}
\left(
  \begin{array}{cc}
    1 &\frac{r_{0}}{1+|r_{0}|^{2}} \\
    0 & 1 \\
  \end{array}
\right)
$};
\end{tikzpicture}}
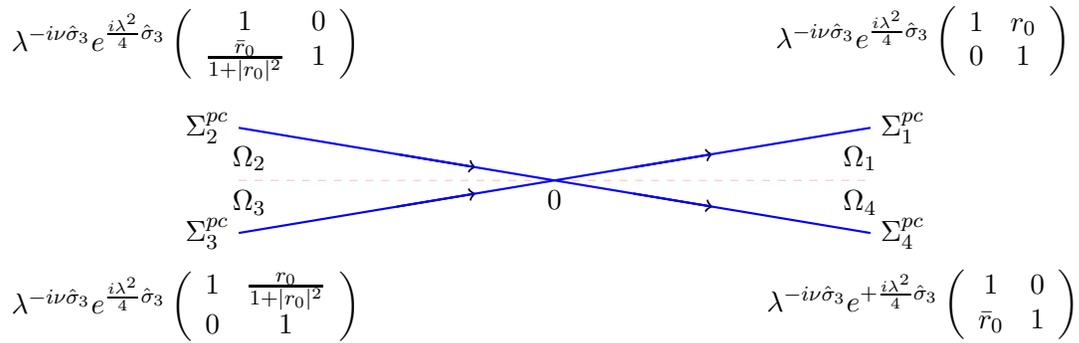
\centerline{\noindent {\small \textbf{Figure 5.} Jump matrix $V^{(pc)}$}.}
We have the parabolic cylinder equation  expressed as \cite{PC-equation}
\begin{align*}
\left(\frac{\partial^{2}}{\partial z^{2}}+(\frac{1}{2}-\frac{z^{2}}{2}+a)\right)D_{a}=0.
\end{align*}
As shown in the literature\cite{Deift-1993, PC-solution2}, we obtain the explicit solution $M^{(pc)}(\lambda, r_{0})$:
\begin{align*}
M^{(pc)}(\lambda, r_{0})=\Phi(\lambda, r_{0})\mathcal{P}(\lambda, r_{0})e^{\frac{i}{4}\lambda^{2}\sigma_{3}}\lambda^{-i\nu\sigma_{3}},
\end{align*}
where
\begin{align*}
\mathcal{P}(\lambda, r_{0})=\left\{\begin{aligned}
&\left(
                    \begin{array}{cc}
                      1 & -r_{0} \\
                      0 & 1 \\
                    \end{array}
                  \right),\quad &\lambda\in\Omega_{1},\\
&\left(
                    \begin{array}{cc}
                      1 & 0\\
                      -\frac{\bar{r}_{0}}{1+|r_{0}|^{2}} & 1 \\
                    \end{array}
                  \right),\quad &\lambda\in\Omega_{2},\\
&\left(
                    \begin{array}{cc}
                      1 &\frac{r_{0}}{1+|r_{0}|^{2}}\\
                      0 & 1 \\
                    \end{array}
                  \right),\quad &\lambda\in\Omega_{3},\\
&\left(
                    \begin{array}{cc}
                      1 & 0 \\
                      \bar{r}_{0} & 1 \\
                    \end{array}
                  \right),\quad &\lambda\in\Omega_{4},\\
&~~~\mathbb{I}, \quad &\lambda\in elsiwhere,
\end{aligned}\right.
\end{align*}
and
\begin{align*}
\Phi(\lambda, r_{0})=\left\{\begin{aligned}
\left(
                    \begin{array}{cc}
                      e^{-\frac{3\pi\nu}{4}}D_{i\nu}\left( e^{-\frac{3i\pi}{4}}\lambda\right) & i\beta_{12}e^{-\frac{3\pi(\nu+i)}{4}}D_{i\nu-1}\left( e^{-\frac{3i\pi}{4}}\lambda\right) \\
                      -i\beta_{21}e^{-\frac{\pi}{4}(\nu-i)}D_{-i\nu-1}\left( e^{-\frac{i\pi}{4}}\lambda\right) & e^{\frac{\pi\nu}{4}}D_{-i\nu}\left( e^{-\frac{i\pi}{4}}\lambda\right) \\
                    \end{array}
                  \right),\quad \lambda\in\mathbb{C}^{+},\\
\left(
                    \begin{array}{cc}
                      e^{\frac{\pi\nu}{4}}D_{i\nu}\left( e^{\frac{i\pi}{4}}\lambda\right) & i\beta_{12}e^{\frac{\pi}{4}(\nu+i)}D_{i\nu-1}\left( e^{\frac{i\pi}{4}}\lambda\right) \\
                      -i\beta_{21}e^{-\frac{3\pi(\nu-i)}{4}}D_{-i\nu-1}\left( e^{\frac{3i\pi}{4}}\lambda\right) & e^{-\frac{3\pi\nu}{4}}D_{-i\nu}\left( e^{\frac{3i\pi}{4}}\lambda\right) \\
                    \end{array}
                  \right),\quad \lambda\in\mathbb{C}^{-},
\end{aligned}\right.
\end{align*}
with
\begin{align}\label{A-3}
\beta_{21}=\frac{\sqrt{2\pi}e^{i\pi/4}e^{-\pi\nu/2}}{r_0\Gamma(-i\nu)},\quad \beta_{12}=\frac{-\sqrt{2\pi}e^{-i\pi/4}e^{-\pi\nu/2} }{\overline{r_0}\Gamma(i\nu)}=\frac{\nu}{\beta_{21}}.\tag{A.6}\\
\arg(\beta_{21})=-\frac{\pi}{2}\nu+\frac{\pi}{4}i-\arg(r_0)-\arg\Gamma(-i\nu(r_0)).\notag
\end{align}
Then, it is not hard to obtain the asymptotic behavior of the solution by using the well-known asymptotic behavior of $D_{a}(z)$,
\begin{align}\label{A-1}
M^{(pc)}(r_0,\lambda)=I+\frac{M_1^{(pc)}}{i\lambda}+O(\lambda^{-2}), \tag{A.7}
\end{align}
where
\begin{align}\label{A-2}
M_1^{(pc)}=\begin{pmatrix}0&-\beta_{12}\\\beta_{21}&0\end{pmatrix}. \tag{A.8}
\end{align}

\section*{Appendix B: Detailed calculations for the pure $\bar{\partial}$-Problem  }

\noindent \textbf{Proposition B.1}
For large $|t|$, there exist constants $c_{j}(j=1,2,3)$ such that $I_{j}(j=1,2,3)$  defined in \eqref{7-8} and \eqref{7-9} possess the following estimate
\begin{align}\label{B-1}
I_{j}\leq c_{j}t^{-\frac{1}{4}},~~ j=1,2,3. \tag{B.1}
\end{align}

\begin{proof}
Let $s=p+iq$ and $z=\xi+i\eta$. Considering the fact that
\begin{align*}
\Big|\Big|\frac{1}{s-z}\Big|\Big|_{L^{2}}(q+z_{0})=(\int_{q+z_{0}}^{\infty}\frac{1}{|s-z|^{2}}dp)^{\frac{1}{2}}
\leq\frac{\pi}{|q-\eta|},
\end{align*}
we can derive that
\begin{align}\label{C-2}
\begin{split}
|I_{1}|&\leq\int_{0}^{+\infty}\int_{q+z_{0}}^{+\infty}
\frac{|\bar{\partial}\chi_{\mathcal{Z}}(s)|e^{-\frac{2|t|}{\Phi_0}q(p-z_{0})}}{|s-z|}dpdq\\
&\leq\int_{0}^{+\infty}e^{-\frac{2|t|}{\Phi_0}q^{2}}\big|\big|\bar{\partial}\chi_{\mathcal{Z}}(s)\big|\big|_{L^{2}(q+z_{0})}
\Big|\Big|\frac{1}{s-z}\Big|\Big|_{L^{2}(q+z_{0})}dq \\
&\leq c_{1}\int_{0}^{+\infty}\frac{e^{-\frac{2|t|}{\Phi_0}q^{2}}}{\sqrt{|q-\eta|}}dq
\leq c_{1}|t|^{-\frac{1}{4}}.
\end{split}\tag{B.2}
\end{align}
Similarly, considering that $r\in H^{1,1}(\mathbb{R})$, we obtain the estimate
\begin{align}\label{C-3}
|I_{2}|\leq\int_{0}^{+\infty}\int_{q+z_{0}}^{+\infty}
\frac{|r'(p)|e^{-\frac{2|t|}{\Phi_0}q^{2}}}{|s-z|}dpdq
\leq c_{2}|t|^{-\frac{1}{4}}.\tag{B.3}
\end{align}
To obtain the estimate of $I_{3}$, we consider the following $L^{k}(k>2)$ norm
\begin{align}\label{C-4}
\bigg|\bigg|\frac{1}{\sqrt{|s-z_{0}|}}\bigg|\bigg|_{L^{k}}
\leq \left(\int_{q+z_{0}}^{+\infty}
\frac{1}{|p-z_{0}+iq|^{\frac{k}{2}}}dp\right)^{\frac{1}{k}}
\leq cq^{\frac{1}{k}-\frac{1}{2}}.\tag{B.4}
\end{align}
Similarly, we can derive that
\begin{align}\label{C-5}
\bigg|\bigg|\frac{1}{|s-z|}\bigg|\bigg|_{L^{k}}\leq c|q-\eta|^{\frac{1}{k}-\frac{1}{2}}.\tag{B.5}
\end{align}
By applying \eqref{C-4} and \eqref{C-5}, it is not hard to check that
\begin{align}\label{C-6}
\begin{split}
|I_{3}|&\leq\int_{0}^{+\infty}\int_{q}^{+\infty}
\frac{|z-z_{0}|^{-\frac{1}{2}}e^{-\frac{2|t|}{\Phi_0}q(p-z_{0})}}{|s-z|}dpdq\\
&\leq\int_{0}^{+\infty}e^{-\frac{2|t|}{\Phi_0}q^{2}}\bigg|\bigg|\frac{1}{\sqrt{|s-z_{0}|}}\bigg|\bigg|_{L^{k}}
\bigg|\bigg|\frac{1}{|s-z|}\bigg|\bigg|_{L^{k}}dq \leq c_{3}t^{-\frac{1}{4}}.
\end{split}\tag{B.6}
\end{align}
Now,   the estimates of $I_{j}(j=1,2,3)$ are proved completely.
\end{proof}

\noindent\textbf{Compliance with ethical standards}\\

\noindent\textbf{Conflict of interest} The authors declare that they have no conflict of interest.


\bibliographystyle{plain}

\end{document}